\documentclass[reqno]{amsart}
\usepackage{amsmath}
\usepackage{amsthm,amssymb, mathtools, mathrsfs, enumerate}
\usepackage{tikz-cd}

\usepackage {hyperref}

\usepackage{xcolor}

\usepackage[margin=1in]{geometry}
\usepackage[T1]{fontenc}
\usepackage[utf8]{inputenc}
\usepackage[semibold]{sourceserifpro}
\usepackage{microtype}
\usepackage{caption} 
\usepackage[capitalize]{cleveref}
\crefname{equation}{}{}
\crefname{enumi}{}{}
\creflabelformat{enumi}{(#2#1#3)}

\theoremstyle{definition}
\newtheorem{dfn}{Definition}[section]

\theoremstyle{plain}
\newtheorem{thm}[dfn]{Theorem}
\newtheorem{thmdfn}[dfn]{Theorem/Definition}
\newtheorem{prop}[dfn]{Proposition}
\newtheorem{lem}[dfn]{Lemma}
\newtheorem{cor}[dfn]{Corollary}
\newtheorem{conj}[dfn]{Conjecture}
\newtheorem*{question}{Problem}

\theoremstyle{remark}
\newtheorem{zbbackend}[dfn]{Example}
\newenvironment{zb}[1][]{\begin{zbbackend}[#1]}{\hspace*{\fill}$\diamondsuit$\end{zbbackend}}
\newtheorem{rmk}[dfn]{Remark}

\newenvironment{claim}{\par\medskip\noindent\textit{Claim.}\space}{\par\medskip}
\newenvironment{claimproof}{\par\noindent\textit{Proof of claim.}\space}{\hfill$\diamond$\medskip\par}

\newcommand{\word}[1]{\textbf{#1}} 
\newcommand{\revision}[1]{{#1}} 

\title{Combinatorial flats and Schubert varieties of subspace arrangements}

\author{Colin Crowley}
\email{crowley@uoregon.edu}
\address{Department of Mathematics, University of Oregon, Eugene, OR 97403
and School of Mathematics, Institute for Advanced Study, Princeton, NJ 08540}
\author{Connor Simpson}
\email{connorgs@connorgs.net}
\author{Botong Wang}
\email{wang@math.wisc.edu}
\address{School of Mathematics,
  Institute for Advanced Study,
  1 Einstein Drive,
  Princeton, NJ 08540
}

\DeclareMathOperator{\Aut}{Aut} 
\DeclareMathOperator{\GL}{GL} 
\DeclareMathOperator{\PGL}{PGL} 
\renewcommand{\P}{\mathbb{P}} 
\newcommand{\A}{\mathbb{A}} 
\newcommand{\Ga}{\mathbb{G}_a} 
\newcommand{\Gm}{\mathbb{G}_m} 
\newcommand{\G}{\mathbb{G}} 
\newcommand{\geo}{\mathrm{geo}}
\newcommand{\simp}{\mathrm{sim}}
\DeclareMathOperator{\St}{St} 

\newcommand{\C}{\mathbb{C}} 
\newcommand{\Q}{\mathbb{Q}} 
\newcommand{\Z}{{\mathbb{Z}}} 
\newcommand{\N}{\mathbb{Z}_{\geq 0}} 
\newcommand{\K}{\mathbb{K}} 

\newcommand{\pt}{\mathrm{pt}} 

\newcommand{\IH}{I\!H}

\DeclareMathOperator{\rk}{rk} 
\DeclareMathOperator{\nul}{null} 
\DeclareMathOperator{\crk}{crk} 

\DeclareMathOperator{\codim}{codim}

\DeclareMathOperator{\Id}{Id}

\newcommand{\ph}{\varphi}

\begin{document}
\begin{abstract}
  The lattice of flats $\mathcal L_M$ of a matroid $M$ is combinatorially well-behaved and, when $M$ is realizable, admits a geometric model in the form of a ``Schubert variety of hyperplane arrangement''.
  In contrast, the lattice of flats of a polymatroid exhibits many combinatorial pathologies and admits no similar geometric model.

  We address this situation by defining the lattice $\mathcal L_P$ of ``combinatorial flats'' of a polymatroid $P$.
  Combinatorially, $\mathcal L_P$ exhibits good behavior analogous to that of $\mathcal L_M$: it is graded, determines $P$ when $P$ is simple, and is top-heavy.
  When $P$ is realizable over a field of characteristic 0, we show that $\mathcal L_P$ is modeled by ``the Schubert variety of a subspace arrangement''.

  Our work generalizes a number of results of Ardila-Boocher and Huh-Wang on Schubert varieties of hyperplane arrangements; however,
  the geometry of Schubert varieties of subspace arrangements is noticeably more complicated than that of  Schubert varieties of hyperplane arrangements.
  Many natural questions remain open.
\end{abstract}
\maketitle
\section{Introduction}
\subsection{} Let $M$ be a matroid of rank $d$ on an $N$-element set $E$, and let $\mathcal L_M$ be its poset of flats. Write $\mathcal L_M^k$ for the set of flats of rank $k$.
Beloved combinatorial properties of $\mathcal L_M$ include:
\begin{itemize}
\item $\mathcal L_M$ is a semimodular lattice (in particular, it is graded), and
\item there is a unique simple matroid $M^\simp$ with $\mathcal L_{M^\simp} \cong \mathcal L_M$.
\item $|\mathcal L_M^k| \leq |\mathcal L_M^{k+1}|$ when $k \leq d/2n$.
\item $\mathcal L_M$ is \word{top-heavy}: $|\mathcal L_M^k| \leq |\mathcal L_M^{d-k}|$ when $k \leq d/2$.
\end{itemize}
Top-heaviness was long conjectured \cite{DW74, DW75}, but proved only recently in \cite{BHMPW20b}.
The approach of \cite{BHMPW20b} is inspired by \cite{HW17}, who proved top-heaviness when $M$ is the matroid associated to a linear subspace $V \subset \K^N$, with $\K$ a field.
Central to \cite{HW17} is the \word{Schubert variety of hyperplane arrangement} $Y_V$, defined in \cite{AB16} as the closure of $V \subset \K^N$ in $(\P^1)^N = (\K \cup \infty)^N$.

The Schubert variety of hyperplane arrangement $Y_V$ may be thought of as an algebro-geometric counterpart of $\mathcal L_M$, fully capturing the combinatorics of $M$.
To see why, first observe that the additive action of $\K$ \revision{on itself} extends to an action on $\P^1$, fixing $\infty$.
Hence, \revision{$\K^N$ acts on $(\P^1)^N$.}
The orbits of this action are of the form
\[ \textstyle
  O_S = \prod_{i \in S} \K \times \prod_{j \in E \setminus S} \infty,
\]
for each $S \subset \{1, \ldots, N\}$.
From \cite{AB16,HW17}, we learn that $Y_V \cap O_S$ is nonempty if and only if $S$ is a flat. In particular, $Y_V$ is a union of finitely many $V$-orbits. Moreover, if $S$ and $S'$ are flats, then
\begin{itemize}
\item  $Y_V \cap O_S$ is a single $V$-orbit, isomorphic to $\rk_M(S)$-dimensional affine space,
\item  $\overline{Y_V \cap O_S} \supset Y_{V} \cap O_{S'}$ if and only if $S \supset S'$, and
\item the class of $\overline{Y_V \cap O_S}$ in the Chow ring of $(\P^1)^n$ is
  \[
    [\overline{Y_V \cap O_S}] = \sum_{B \text{ basis of $S$}} {\textstyle [(\P^1)^B],}
  \]
  where $(\P^1)^B = \prod_{i \in B} \P^1 \times \prod_{j \in E \setminus B} \infty$.
\end{itemize}  

\medskip
\subsection{} \label{intropolymatroid}
Polymatroids are to subspace arrangements as matroids are to hyperplane arrangements.
If $M$ is replaced by a rank-$d$ polymatroid $P$ on $E$, then all expected combinatorial properties fail.
The poset of flats of a polymatroid is always a lattice, but may not be graded.
When $P$ is simple, the poset fails to determine $P$, and there is no sense in which it is top-heavy.

\revision{To remedy these pathologies, we introduce the \word{lattice of combinatorial flats} $\mathcal L_P$ of a polymatroid.
When $P$ is a matroid, $\mathcal L_P$ is the usual lattice of flats.
When $P$ is realized by a subspace arrangement $V_1, \ldots, V_N$, the lattice $\mathcal L_P$ is the poset of flats of a collection of general flags with minimal elements $V_1, \ldots, V_N$ (see \cref{intro-example}).

For arbitrary polymatroids,  $\mathcal L_P$ may be constructed as follows\footnote{The definitions of ``rank'' and ``combinatorial flat'' given here are different from, but equivalent to, the definitions given in \cref{comboflatsminimum}. The equivalence is explained in \cref{equivalence}.}.
Let $\rk_P$ be the rank function of $P$, and let $\mathbf e_1, \ldots, \mathbf e_N$ be the standard basis of $\Z^N$.
Choose $\mathbf n = (n_1,\ldots, n_N) \in \N^N$, such that $n_i \geq \rk_P(i)$.
Recall that an \word{independent multiset} of $P$ is $\mathbf b = (b_1, \ldots, b_N) \in \N^N$ such that $\sum_{i \in A} b_i \leq \rk_P(A)$ for all $A \subset E$.
A \word{basis} for a multiset $\mathbf s = (s_1, \ldots, s_N) \in \N^N$ is an independent multiset $\mathbf b = (b_1, \ldots, b_N)$, maximal among those contained in $\mathbf s$.
Define the \word{rank} of $\mathbf s$ by $\rk_P(\mathbf s) := \sum_{i \in E} b_i$ for any basis $\mathbf b$\footnote{\cref{welldfndrk} explains well-definedness of rank. We write $\rk_P(\cdot)$ for ranks of both sets and multisets; this will never cause confusion.}, and
call $\mathbf s$ a \word{combinatorial flat}\footnote{This definition resembles the related notion of ``flats for matricubes'' \cite{AG24}. See \cref{matricube} for comparison of these notions.}
if $\rk_P(\mathbf s + \mathbf e_i) > \rk_P(\mathbf s)$ for all $i$ such that $s_i < n_i$.

We define $\mathcal L_P$ to be the set of combinatorial flats, ordered by inclusion.
We will see in \cref{comboflatsimplification} that $\mathcal L_P$ does not depend on $\mathbf n$ (although the underlying set of combinatorial flats does; we write $\mathcal L_{P, \mathbf n}$ when we wish to emphasize this).}

\begin{thm}\label{combo}
  Let $P$ be a polymatroid. \revision{The poset of combinatorial flats satisfies:}
  \begin{enumerate}
  \item \label{combo:semimodularlattice} $\mathcal L_P$ is a graded semimodular lattice,
  \item \label{combo:simple} there is a unique simple polymatroid $P^\simp$ such that $\mathcal L_P \cong \mathcal L_{P^\simp}$,
  \item $|\mathcal L_P^k| \leq |\mathcal L_P^{k+1}|$ when $k \leq d/2$, and
  \item \label{combo:topheavy} $\mathcal L_P$ is top-heavy: $|\mathcal L_P^k| \leq |\mathcal L_P^{d-k}|$ when $k \leq d/2$.
  \end{enumerate}
\end{thm}

\begin{zb}\label{intro-example}
  Suppose that $V_1,V_2,V_3,V_4 \subseteq V = \C^3$ is a linear subspace arrangement, where $V_1,V_2,$ and $V_3$ are two dimensional, $V_1 \cap V_2 \cap V_3$ is one dimensional, and $V_4$ is one dimensional and not contained in $V_i$ for $i=1,2,3$.
  This arrangement defines the polymatroid $P$ on $E = \{1,2,3,4\}$ with $\rk_P(S) = 3 - \dim \cap_{i \in S} V_i$.
  Flats of $P$ correspond to intersections of $V_i$'s. The lattice of flats is neither graded nor top-heavy  (\cref{intro-picture}).
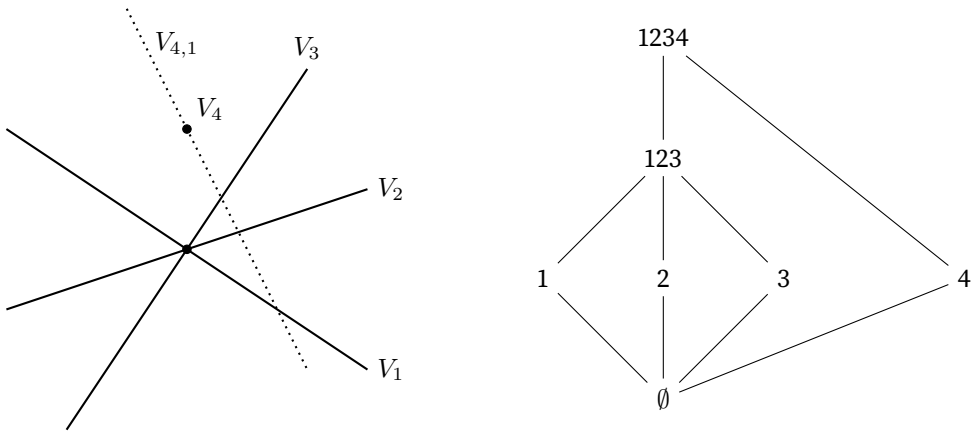
\begin{figure}[h!]
    \hspace{-1cm}
    \begin{minipage}{3in}
      \centering
    \begin{tikzpicture}[scale=0.8]
    \draw[thick] (-3,-1) -- (3,1) node[pos=1,right] {$V_2$};  
    \draw[thick] (-3,2) -- (3,-2) node[pos=1,right] {$V_1$};  
    \draw[thick] (-2,-3) -- (2,3) node[pos=1,above] {$V_3$};  

    \filldraw (0,0) circle (2pt);
    
    \filldraw (0,2) circle (2pt);
  \node[above right] at (0,2) {$V_4$};  
    
    \draw[dotted, thick] (-1,4) -- (2,-2) node[pos=0.1,right] {$V_{4,1}$};  
  \end{tikzpicture}

      \end{minipage}
      \quad
      \begin{minipage}{2.5in}
\begin{tikzpicture}[scale=1.6]
  \node (1234) at (0,5) {1234};
  \node (123) at (0,4) {123};
  \node (1) at (-1,3) {1};
  \node (2) at (0,3) {2};
  \node (3) at (1,3) {3};
  \node (4) at (2.5,3) {4};
  \node (empty) at (0,2) {$\emptyset$};

  \draw (empty) -- (1);
  \draw (empty) -- (2);
  \draw (empty) -- (3);
  \draw (empty) -- (4);
  \draw (1) -- (123);
  \draw (2) -- (123);
  \draw (3) -- (123);
  \draw (4) -- (1234);
  \draw (123) -- (1234);
\end{tikzpicture}
\end{minipage}
        \caption{Left: the projectivization of the subspaces arrangement $V_1, \ldots, V_4$ from \cref{intro-example} (solid), along with the projectivization of a general plane $V_{4,1}$ containing $V_4$ (dashed). Right: the lattice of flats of the polymatroid $P$ realized by $\{V_1,V_2,V_3,V_4\}$.}
        \label{intro-picture}
      \end{figure}


\revision{On the other hand, $\mathcal L_P \cong \mathcal L_{P, \mathbf n}$ (with $\mathbf n = (1,1,1,2)$) is both graded and top-heavy (\cref{intro-diagram}).}
Elements of $\mathcal L_P$ correspond to flats of the arrangement $\{V_1, V_2,V_3,V_4,V_{4,1}\}$, where $V_{4,1}$ is a general plane containing $V_4$. There is a lattice homomorphism embedding the lattice of flats of $P$ into $\mathcal L_P$.
This construction is explained in full generality in \cref{geometryofcf}.
\begin{figure}[h!]
  \begin{center}
    \newcommand{\joinirred}[1]{\underline{#1}}
    \newcommand{\latticeimage}[1]{{\color{blue} #1}}
    \begin{tikzcd}
                                                                             &  &                                                                   & {\latticeimage{\mathbf{e}_1 + \mathbf{e}_2 + \mathbf{e}_3 + 2\mathbf{e}_4}}                                                                         &                                                                    &  &                                                                                                   \\
\latticeimage{\mathbf{e}_1 + \mathbf{e}_2 + \mathbf{e}_3} \arrow[rrru, no head] &  & \mathbf{e}_1 + \mathbf{e}_4 \arrow[ru, no head]                   & \mathbf{e}_2 + \mathbf{e}_4 \arrow[u, no head]                                                             & \mathbf{e}_3 + \mathbf{e}_4 \arrow[lu, no head]                    &  & \latticeimage{\joinirred{2\mathbf{e}_4}} \arrow[lllu, no head]                                                   \\
\latticeimage{\joinirred{\mathbf{e}_1}} \arrow[u, no head] \arrow[rru, no head]             &  & \latticeimage{\joinirred{\mathbf{e}_2}} \arrow[llu, no head] \arrow[ru, no head] &                                                                                                            & \latticeimage{\joinirred{\mathbf{e}_3}} \arrow[llllu, no head] \arrow[u, no head] &  & \joinirred{\mathbf{e}_4} \arrow[llllu, no head] \arrow[lllu, no head] \arrow[llu, no head] \arrow[u, no head] \\
                                                                             &  &                                                                   & \latticeimage{\mathbf{0}} \arrow[lu, no head] \arrow[ru, no head] \arrow[rrru, no head] \arrow[lllu, no head] &                                                                    &  &                                                                                                  
\end{tikzcd}
\end{center}
\caption{\label{intro-diagram}The lattice of combinatorial flats of $P$ from \cref{intro-example}. Join-irreducibles are underlined. Blue elements are in the image of the embedding of $P$'s lattice of flats into $\mathcal L_P$.}
\end{figure}
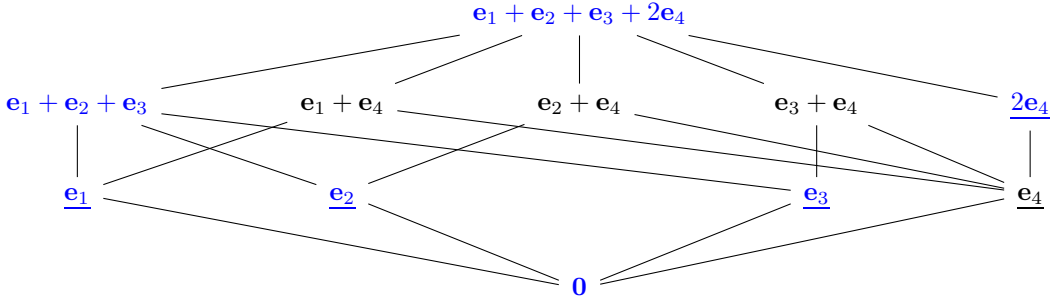

\vspace{-0.7cm}
\qedhere
\end{zb}

\revision{Lattices of combinatorial flats can be axiomatized in a manner that closely resembles the well-known definition of geometric lattices (\cref{latticeaxiom}). Consequently, they constitute a new cryptomorphic definition of simple polymatroids.}

\subsection{} Unlike the usual poset of flats of $P$, $\mathcal L_P$ can be realized as the poset of affine cells of a projective variety whenever $P$ is realizable \revision{over a field of characteristic zero}.  

Work of Hassett-Tschinkel \cite{HT99} (reviewed in \cref{sec_HT}) shows that there exists action of the additive group $\K^{\mathbf n} := \K^{n_1} \times \cdots \times \K^{n_N}$ on $\P^{\mathbf n} := \P^{n_1} \times \cdots \times \P^{n_N}$, such that 
$\P^{\mathbf n}$ is partitioned into finitely many orbits $O_{\mathbf s}$, corresponding to multisets $\mathbf s \leq \mathbf n$.
Under the assumption that $V$ is ``polymatroid general'', a mild genericity condition described in \cref{pgprops} that constrains the intersection of $V$ with certain coordinate subspaces, we define $Y_V$ to be the $V$-orbit closure of the origin in $\P^{\mathbf n}$ under the Hasset-Tschinkel action of $\K^{\mathbf n}$ on $\P^{\mathbf{n}}.$
\begin{thm}\label{main:stratification}
  If $V \subset \K^{\mathbf n}$ is a polymatroid general subspace realizing $P$, then $Y_V \cap O_{\mathbf s}$ is nonempty if and only if $\mathbf s$ is a combinatorial flat.
  Moreover, if $\mathbf s, \mathbf s' \in \mathcal L_P \cong \mathcal L_{P, \mathbf n}$, then
  \begin{enumerate}
  \item \label{main:stratification:orbit} $Y_V \cap O_{\mathbf s}$ is a single $V$-orbit, isomorphic to affine space of dimension $\rk_P(\mathbf s)$,
  \item \label{main:stratification:poset} $\overline{Y_V \cap O_{\mathbf s}} \supset Y_V \cap O_{\mathbf s'}$ if and only if $\mathbf s \geq \mathbf s'$, and
  \item \label{main:stratification:class} the class of $\overline{Y_V \cap O_{\mathbf s}}$ in the Chow ring of $\P^{\mathbf n}$ has the form
    \[
      [\overline{Y_V \cap O_{\mathbf s}}] = \sum_{\mathbf b \text{ basis of $\mathbf s$}} c_{\mathbf b} [\P^{\mathbf b}],
    \]
    with all coefficients positive.
  \end{enumerate}
\end{thm}

If $\K = \C$, we also give a presentation for the singular cohomology ring of $Y_V$, generalizing a result of \cite{HW17}.

\revision{\begin{dfn}
    Let $P$ be a polymatroid with ground set $E=\{1, \ldots, N\}$. Given two combinatorial flats $\mathbf{s}$ and $\mathbf{s'}$ spanned by independent multisets $\mathbf{b}$ and $\mathbf{b'}$ respectively, we define their \textbf{cup product} as
    \[
    (\mathbf{s}\smile\mathbf{s'}):=\overline{\left(\min\{\mathbf{b}(i)+\mathbf{b'}(i), \rk_P(i)\}\right)_{i=1}^N}.
    \]
\end{dfn}
\begin{rmk}
    Equivalently, in terms of Definition \ref{defnCombFlat}, $\mathbf{s}\smile\mathbf{s'}$ is the combinatorial flat represented by any maximal element of $\{F \vee F'\}$ where $F$ and $F'$ range over flats of the multisymmetric lift $\widetilde{P}$ which represent $\mathbf{s}$ and $\mathbf{s'}$ respectively. Notice that the product of symmetric groups $\mathfrak S := \prod_{i=1}^N \mathfrak S_{E_i}$ acts transitively on the maximal elements of $\{F \vee F'\}$. Therefore $\mathbf{s}\smile\mathbf{s'}$ is well defined, independent of the choices of $\mathbf{b}$ and $\mathbf{b'}$.
\end{rmk}
}
    
\revision{\begin{zb}
    Suppose that $N=1$, and $\rk_P(\{1\}) = n$. Then the lattice $\mathcal L_P$ of combinatorial flats is isomorphic to the chain of length $n$. Any combinatorial flat $i \in [n]$ is also independent in $P$, so the cup product of combinatorial flats is given as
    \[
        i \smile j := \overline{\min(i+j, n)} = \min(i+j,n).
    \]
\end{zb}}

\begin{thm}\label{cohomology2}
  If $V \subset \C^{\mathbf n}$ is a polymatroid general subspace, then the
  singular cohomology ring of $Y_V$ is the graded vector space
  \[ H^*(Y_V; \Q) \cong \bigoplus_{\mathbf s \in \mathcal L_{P}} \Q \cdot y_{\mathbf s},
  \]
  where $\deg y_{\mathbf s} = \rk_P(\mathbf s)$. There are \revision{positive} constants $c_{\mathbf b}$ associated to bases $\mathbf b$ of combinatorial flats such that  multiplication is given by the formula
  \[
    c_{\mathbf b} c_{\mathbf b'} y_{\mathbf s} y_{\mathbf s'} = \begin{cases}
                                     c_{\mathbf b + \mathbf b'} y_{\mathbf s \smile \mathbf s'}, & \text{if $\rk_P(\mathbf s \smile \mathbf s') = \rk_P(\mathbf s) + \rk_P(\mathbf s')$} \\
     0, & \text{otherwise,}                                
    \end{cases}
  \]
  where $\mathbf b$ and $\mathbf b'$ are bases for $\mathbf s$ and $\mathbf s'$, respectively, such that $\mathbf b + \mathbf b'$ is a basis for $\mathbf s \smile \mathbf s'$.
\end{thm}

\subsection{Organization}
We review polymatroids, multisymmetric matroids, and operations on them in \cref{sec_poly}, before commencing study of combinatorial flats in \cref{comboflats}.

We construct combinatorial flats in \cref{comboflatsminimum}, then consider the effects of deletion and truncation in \cref{comboflatsdeletiontruncation}. These sections are the only portions of \cref{comboflats} required for \cref{sec:polymatroidschubert}.
In the remainder of \cref{comboflats},  we work up to a proof of \cref{combo} (Sections \ref{posetprops} and \ref{comboflatsimplification}) and provide axioms for $\mathcal L_P$ (\cref{comboflataxioms})

\cref{sec:polymatroidschubert} concerns Schubert varieties of subspace arrangements.
Foundational work of Hassett and Tschinkel \cite{HT99} is reviewed in Sections \ref{sec_HT} and \ref{rmk_Gm}.
We construct $Y_V$ in \cref{polymatschubertdfn}, then compute the support of its Chow class in \cref{chowclass}.
In \cref{pgprops}, we discuss ``polymatroid generality'', a condition on $V$ that guarantees $Y_V$ consists of finitely many $V$-orbits.
Assuming this condition, we identify strata of $Y_V$ in \cref{geomproofs}, culminating with a proof of \cref{main:stratification}.
In \cref{sec:topresults}, we prove \cref{cohomology2}.

At the last, we pose a large number of questions about Schubert varieties of subspace arrangements in \cref{sec:problems}.
Both algebraic and geometric aspects of Schubert varieties of subspace arrangements seem markedly more complicated than those of Schubert varieties of hyperplane arrangements, yet examples suggest a good theory lies in wait.
Interested parties are encouraged to seek their fortune.

\subsection{Notation}\label{sec:notation}
Throughout this paper, $E := \{1, \ldots, N\}$.
A \word{multiset} on $E$ is an element of $\N^E$. 
The \word{cardinality} of a multiset $\mathbf s = (s_1, \ldots, s_N)$ is $|\mathbf s| := s_1 + \cdots + s_N$.
If $\widetilde E = E_1 \sqcup \cdots \sqcup E_N$ is a finite set, then  $S \subset \widetilde E$ \word{represents}  $\mathbf s$ if $s_i = |S \cap E_i|$ for each $1 \leq i \leq N$.
Multisets are ordered: $(s_1, \ldots, s_N) \leq (s_1', \ldots, s_N')$  if and only if $s_i \leq s_i'$ for every $1 \leq i\leq N$.
Multisets will always be written in boldface.

Let $\K$ be a field. Products of $\K$-vector spaces and projective spaces will be written $\K^{\mathbf s} := \K^{s_1} \times \cdots \times \K^{s_N}$ and $\P^{\mathbf s} := \P^{s_1} \times \cdots \times \P^{s_N}$, respectively.
Like multisets, elements of $\K^{\mathbf s}$ will be written in boldface; we trust this will cause no confusion.

We write $\mathbf e_1, \ldots, \mathbf e_N$ for the standard basis of $\Z^E$,  and $\mathbf 0$ for the zero vector.
We will regard these as multisets or vectors as needed.

Braces for operations involving one-element sets are often omitted, e.g. we write $S \setminus j$ for $S \setminus \{j\}$ and $\pi_j$ for $\pi_{\{j\}}$.
The symmetric group on a set $S$ is denoted $\mathfrak S_S$.

If $(A_i)_{i \in S}$ is a family of sets indexed by a set $S$, and $T \subset S$, then $\pi_T: \prod_{i \in S} A_i \to \prod_{i \in T}A_i$ is the projection.
We will use the same symbol $\pi_T$ for the projection out of any product indexed by a set containing $T$.
For example, we will write $\pi_{E \setminus N}$ both for the projection $\N^N \to \N^{N-1}$ that forgets the last coordinate of a multiset, and for the projection $\K^{\mathbf s} \to \K^{\pi_{E \setminus N}(\mathbf s)}$ that forgets the last coordinate of a vector in $\K^{\mathbf s}$. Once more, we expect this conventions will cause no confusion.

\subsection{Acknowledgements}
We thank Lucas Gierczak for helping to sort out the relationship between the ``matricubes'' of \cite{AG24} and the objects of the present work (see \cref{matricube}), and Matt Larson for helpful comments on our draft.
We are all partially supported by the Simons Foundation. The first author is partially supported by NSF grant DMS-2039316, and the last author is partially supported by NSF grant DMS-1926686. We are grateful to the Institute for Advanced Study (IAS). The second and third authors also thank the Beijing International Center for Mathematical Research (BICMR) for their hospitality and support during the preparation of this paper. 

\section{Polymatroids}\label{sec_poly}
We review polymatroids, the nearly-equivalent notion of multisymmetric matroids, and some operations on them.

\subsection{}

\noindent A \word{polymatroid} $P$ on a finite set $E = \{1, \ldots, N\}$ is the data of a \word{rank} function $\rk: 2^E \to \Z_{\geq 0}$ that is
\begin{description}
\item[\textit{normalized}] $\rk(\emptyset) = 0$,
\item[\textit{increasing}] if $A \subset B$, then $\rk(A) \leq \rk(B)$, and
\item[\textit{submodular}] if $A, B \subset E$, then $\rk(A \cup B) + \rk(A \cap B) \leq \rk(A) + \rk(B)$.
\end{description}
The \word{rank} of $P$ is $\rk(P) := \rk(E)$, and 
the \word{corank} of a set is $\crk(A) := \rk(P) - \rk(A)$.
If $\rk(A) \leq |A|$ for every $A \subset E$, then $P$ is a \word{matroid}.
A \word{flat} is a subset $F \subset E$ such that $\rk(F \cup i) > \rk(F)$ for all $i \in E \setminus F$.
The intersection of two flats is a flat.

The \word{closure} of $A \subset E$ is $\overline A := \cap_{F \supset A} F$, where $F$ runs over all flats of $P$ containing $A$.
The closure of $A$ is the smallest flat that contains $A$, so $\rk(A) = \rk(\overline A)$.
Call $P$ \word{loopless} if each element of $E$ has positive rank, and \word{simple} if each one-element subset of $E$ is a flat of positive rank.

A \word{basis} of $P$ is a multiset $\mathbf b \in \N^E$ satisfying $|\mathbf b| = \rk(P)$ and $|\pi_A(\mathbf b)| \leq \rk(A)$ for all $A \subsetneq E$.
A polymatroid is determined by its bases because $\rk(A) = \max \{ |\pi_A(\mathbf b)| : \mathbf b \text{ is a basis of $P$}\}$.
A multiset $\mathbf s$ is \word{independent} if there is a basis $\mathbf b$ such that $\mathbf s \leq \mathbf b$, equivalently, if $|\pi_A(\mathbf s)| \leq \rk(A)$ for all $A \subset E$.

When necessary, we will use subscripts to distinguish the rank functions of different polymatroids, e.g.  $\rk_P$  is the rank function of $P$.
\begin{zb}[Free polymatroids]\label{zb:free}
  Let $E = \{1, \ldots, N\}$ and let $\mathbf s \in \N^E$ be a multiset.
  The \word{free polymatroid} $B_{\mathbf s}$ has rank function
    $\rk_{B_{\mathbf s}}(A) = \sum_{i \in A} s_i$.
\end{zb}
\begin{zb}[Realizable polymatroids]\label{zb:geo}
  Let $W_1, \ldots, W_N$ be vector spaces.
  A linear subspace $V \subset \prod_{i=1}^N W_i$ defines a polymatroid $P$ on $E = \{1, \ldots, N\}$ with rank function
  \[ \rk(A) := \dim \pi_A(V) = \codim_V V \cap \ker(\pi_A). \]
  In this case, we say $V$ \word{realizes} $P$.
  The flats $F$ of $P$ are in inclusion-reversing bijection with subspaces $V \cap \ker(\pi_F)$.
  The closure of  $A \subset E$ is the largest  $A' \subset E$ such that $V \cap\ker(\pi_A) = V \cap \ker(\pi_{A'})$.

  Likewise, a subspace arrangement $V_1, \ldots, V_N$ in a vector space $V$ \word{realizes} $P$ if $\rk(A) = \codim_V \cap_{i \in A} V_i$.
  These two notions of realizability are equivalent. Given a subspace $V \subset \prod_i W_i$, take $V_i = \ker \pi_i$. Given a subspace arrangement, choose vector spaces $W_1, \ldots, W_N$ and a map $f: V \to \prod_i W_i$ with $\ker(\pi_i \circ f) = V_i$.
  The image of $f$ is a subspace realizing $P$.
\end{zb}
\begin{zb}[Lifts]\label{zb:partition}
  If $M$ is a matroid on a finite set $\widetilde E = E_1 \sqcup \cdots \sqcup E_N$, then we may define a polymatroid $P$ on $E$ by
  $\rk_P(S) := \rk_M(\cup_{i \in S} E_i)$. In this case, we say that $M$ \word{lifts} $P$.
\end{zb}

\subsection{}
Let $\widetilde E = E_1 \sqcup \cdots \sqcup E_N$ be a finite set. A \word{multisymmetric matroid} is a matroid $M$ on $\widetilde E$ whose rank function is invariant under the natural action of $\mathfrak S := \prod_{i=1}^N \mathfrak S_{E_i}$.
If $F$ is a flat of $M$ and $\sigma \in \mathfrak S$, then $\sigma \cdot F$ is also a flat of $M$.
The \word{geometric part} of a subset $A \subset \widetilde E$ is $A^\geo := \cap_{\sigma \in \mathfrak S} \sigma(A)$, and $A$ is \word{geometric} if $A = A^\geo$.
If $F$ is a flat, then $F^\geo$ is also a flat.
Two key properties of multisymmetric matroids follow.
\begin{lem}\cite[Lemma 2.8]{CHLSW22}\label{rankformula}\label{closure}\label{multisymmetricsubtraction}
  Let $M$ be a multisymmetric matroid on $\widetilde E = E_1 \sqcup \cdots \sqcup E_N$.
  The following equivalent statements all hold:
  \begin{itemize}
  \item If $F$ is a flat of $M$, then $\rk(F) = \rk(F^\geo) + |F \setminus F^\geo|$.
  \item If $A \subset \widetilde E$, then for each $1 \leq i \leq N$, either $\overline{A} \cap E_i = A \cap E_i$ or $\overline A \supset E_i$.
  \item If $F$ is a flat of $M$, then $F \setminus e$ is a flat of rank $\rk(F) - 1$ for any $e \in F \setminus F^\geo$.
  \end{itemize}
\end{lem}

\begin{lem}\cite[Lemma 2.9]{CHLSW22}\label{determinedbygeo}
  Two multisymmetric matroids on $E_1 \sqcup \cdots \sqcup E_N$ are isomorphic if and only if their geometric sets have the same ranks.
\end{lem}

\subsection{}
Multisymmetric matroids and polymatroids are related by a much-rediscovered \cite{H72,N86,L77,M75,BCF22,CHLSW22} recipe.
A \word{cage} for a polymatroid $P$ on $E$ is $\mathbf n \in \N^E$ such that $\rk(i) \leq n_i$ for all $i \in E$.
The pair $(P, \mathbf n)$ is a \word{caged polymatroid}\footnote{This terminology originates in \cite{EL23}.}.
\begin{thmdfn}\label{lift}\cite[Theorem 2.10\footnotemark]{CHLSW22}
  \footnotetext{In \cite{CHLSW22}, this is stated in the case when $\rk_P(i) = |E_i|$, but the proof goes through verbatim in the present slightly more general setting.}
  Fix a multiset $\mathbf n = (n_1, \ldots, n_N)$, and let $\widetilde E = E_1 \sqcup \cdots \sqcup E_N$ be a set with $|E_i| = n_i$.
  There is a bijection
  \begin{align*}
    \{\text{caged polymatroids $(P, \mathbf n)$}\} &\overset{\sim}{\longrightarrow} \{\text{multisymmetric matroids on $\widetilde E$}\} \\
    (P, \mathbf n) &\longmapsto \widetilde P,
  \end{align*}
  where $\widetilde P$ is the \word{multisymmetric lift} of $P$, defined by
    \[
    \rk_{\widetilde P}(A) := \min \{\rk_P(B) + |A \setminus \cup_{i \in B} E_i| : B \subset E \}, \quad A \subset \widetilde E.
  \]
\end{thmdfn}
\begin{rmk}\label{rmk:liftclosure}
  The matroid lift satisfies $\rk_P(S) = \rk_{\widetilde P}(\cup_{i \in S} E_i)$.
  By applying this fact and \cref{rankformula}, we learn that the closure in $\widetilde P$ of $\cup_{i \in S} E_i$ is $\cup_{i \in \overline S} E_i$, where $\overline S$ is the closure of $S$ in $P$.
  
\end{rmk}
\begin{rmk}
   By \cref{determinedbygeo}, $\widetilde P$ is the unique multisymmetric matroid on $\widetilde E$ that lifts $P$.
\end{rmk}
It is also helpful to understand $\widetilde P$ in terms of bases.
\begin{prop}\label{liftbases}
  Let $(P, \mathbf n)$ be a caged polymatroid with multisymmetric lift $\widetilde P$.
  The bases of $\widetilde P$ are the multisets $\mathbf b \in \{0,1\}^{\widetilde E}$ with $|\mathbf b| = \rk(P)$ such that for all $A \subset E$,
  \[
    |\pi_{\cup_{i \in A} E_i}(\mathbf b)| \leq \rk_P(A).
  \]
  Equivalently, $B \subset \widetilde E$ is a basis of $\widetilde P$ if and only if $\big(|B \cap E_1|, \ldots, |B \cap E_N|\big)$ is a basis of $P$.
\end{prop}
\begin{proof}
  The equivalence of the two characterizations is apparent from the definitions. We prove the characterization by inequalities.
  Any basis must satisfy the proposed inequalities because $\rk_P(A) = \rk_{\widetilde P}(\cup_{i \in A} E_i)$.
  Conversely, suppose $\mathbf b \in \{0,1\}^{\widetilde E}$ satisfies the proposed inequalities. For any  $A \subset \widetilde E$,
  \begin{align*}
    \rk_{\widetilde P}(A) = \rk_{\widetilde P}(\overline{A})
    &= |\overline A \setminus {\overline A}^\geo| + \rk_{\widetilde P}(\overline A^\geo) & \text{ by \cref{rankformula}}&\\
    & = |\overline A \setminus {\overline A}^\geo| + \rk_P(\{i : E_i \subset \overline A\})  &\\
    &\geq |\overline A \setminus {\overline A}^\geo| + \sum_{j \in {\overline A}^\geo} b_j & \text{by hypothesis}&\\
    &\geq |\pi_{\overline A}(\mathbf b)| \geq |\pi_A(\mathbf b)|, & \text{  so $\mathbf b$ is a basis of $\widetilde P$.} &\qquad \qedhere
  \end{align*}
\end{proof}

\begin{rmk}[Realizing lifts]\label{realizablelift}
  Let $\K$ be a field and fix $\mathbf n \in \N^N$.
  A linear subspace $V \subset \K^{\mathbf n}$ realizes both a polymatroid $P$ on $E = \{1, \ldots, N\}$ and a matroid $M$ on $\widetilde E = \{(i,j) : 1 \leq i \leq N, \, 1 \leq j \leq n_i\}$.
  Different  $\prod_{i=1}^N \GL(n_i)$-translates of $V$ may realize different matroids; however, the matroid of a general translate is $\widetilde P$, and the polymatroid of any translate is $P$.
\end{rmk}
\begin{zb}
  Suppose that $\mathbf n = (3)$ and that $P$ is the polymatroid of rank 2 on one element, realized by a line $V \subset \K^{\mathbf n}$. Every line in $\K^{\mathbf n}$ realizes $P$ because every line is codimension 2. However, different lines can have different matroids, e.g. a general line realizes a uniform matroid of rank 1 on 3 elements, while each of the three coordinate lines realizes a boolean matroid on 1 element with two loops adjoined.
\end{zb}

\medskip
In the remainder of this section, we describe how the matroid lift interacts with three operations on polymatroids.
Throughout, we take $(P, \mathbf n)$ to be a caged polymatroid on $E = \{1, \ldots, N\}$, and $\widetilde P$ to be the multisymmetric lift of $(P, \mathbf n)$, on ground set $\widetilde E = E_1 \sqcup \cdots \sqcup E_N$.
\subsection{}\label{operations}
The \word{deletion} of $A \subset E$ is the polymatroid $P \setminus A$ obtained by restricting $\rk_P$ to subsets of $E \setminus A$.
The \word{restriction} of $P$ to $A$ is $P|_A := P \setminus (E \setminus A)$.
We were unable to find a reference for the following well-known description of the flats of $P \setminus A$.
\begin{lem}\label{deletionflats}
  The flats of $P \setminus A$ are all sets of the form $F \setminus A$ such that $F$ is a flat of $P$.
\end{lem}
\begin{proof}
    A set $G \subset E\setminus A$ is a flat of $P \setminus A$ if and only if for all $i \in E \setminus A$, $\rk_P(G \cup i) > \rk_P(G)$.
  The latter condition is equivalent to saying the closure $F$ of $G$ in $P$ contains no elements of $E \setminus A$, which is equivalent to saying $G = F \setminus A$.
\end{proof}
 Deletion commutes with lifts.
\begin{lem}\label{deletioncommutes} \cite[Lemma 2.14{\footnotemark[\value{footnote}]{}}]{CHLSW22}
 For any subset $A\subset E$, 
 \[
 \widetilde{P \setminus A} = \widetilde P \setminus \cup_{i \in A} E_i,
 \]
where the left-hand lift is taken with respect to $\pi_{E \setminus A}(\mathbf n)$.
\end{lem}
\begin{rmk}[Realizing deletions]
  With notation as in \cref{zb:geo}, if $P$ is realized by $V$, then $P \setminus A$ is realized by $\pi_{E \setminus A}(V)$.
\end{rmk}
\subsection{}\label{sec:truncation} The \word{truncation} of $P$ at a set $S \subset E$ is the polymatroid $T_S P$ on $E$, with rank function
\[
  \rk_{T_SP}(A) :=
  \begin{cases}
    \rk_P(A) - 1, & \text{if $\rk_P(A) = \rk_P(A \cup S)$} \\
    \rk_P(A), & \text{otherwise.}
  \end{cases}
\]
Since $\rk_P(A \cup S) = \rk_P(\overline{A \cup S}) = \rk_P(A \cup \overline S)$, $T_S P = T_{\overline S} P$.
\begin{lem}\label{truncationcommutes}
  If $S \subset E$, then 
  \[
  \widetilde{T_S P} = T_{\cup_{i \in S} E_i} \widetilde P.
  \]
\end{lem}
\begin{proof}
    Since $\cup_{i \in S} E_i$ is a geometric set of $\widetilde P$, the truncation $T_{\cup_{i \in S} E_i} \widetilde P$ is also a multisymmetric matroid on $\widetilde E$.
  For $A \subset E$,
\begin{align*}
    \rk_{T_{\cup_{i \in S} E_i} \widetilde P}(\cup_{i \in A} E_i)
    &=
    \begin{cases}
      \rk_{\widetilde P}(\cup_{i \in A} E_i) - 1, & \text{if $\rk_{\widetilde P}(\cup_{i \in A} E_i) = \rk_{\widetilde P}(\cup_{i \in A\cup S} E_i)$} \\
      \rk_{\widetilde P}(\cup_{i \in A} E_i), & \text{otherwise}
    \end{cases} \\
    &= \begin{cases}
      \rk_{P}(A) - 1, & \text{if $\rk_{P}(A) = \rk_{P}(A \cup S)$} \\
      \rk_{P}(A), & \text{otherwise}
      \end{cases}\\
    &= \rk_{T_SP}(A).
\end{align*}
  Consequently, $\widetilde{T_S P} = T_{\cup_{i \in S} E_i} \widetilde P$ by \cref{determinedbygeo}.
\end{proof}
  

\begin{rmk}[Realizing truncations]\label{opsgeometry}
  With notation as in \cref{zb:geo}, if $S \subset E$, then $T_S P$ is realized by $V \cap \pi^{-1}_{E \setminus S}(H)$, with $H$ a general hyperplane in $\prod_{i \in S} W_i$.
\end{rmk}

\subsection{}
The \word{reduction} of $P$ at an element $i \in E$ of positive rank is the polymatroid $R_i P$ defined by
\[
  \rk_{R_i P}(A) :=
  \begin{cases} \rk_P(A) - 1, & \rk_P(A) = \rk_P(A \setminus i) + \rk_P(i) \\
     \rk_P(A), & \text{otherwise.}
   \end{cases}
 \]
 
The \word{reduction} of a caged polymatroid is defined by
\[
  R_i(P, \mathbf n) :=
  \begin{cases}
    (P, \mathbf n - \mathbf e_i), & n_i > \rk_P(i) \\
    (R_i P, \mathbf n - \mathbf e_i), & n_i = \rk_P(i).
  \end{cases}
 \]

\begin{lem}\label{reductionlift}
   The multisymmetric lift of $R_i(P, \mathbf n)$ is isomorphic to $\widetilde P \setminus j$, where $j$ is any element of $E_i$.
\end{lem}
\begin{proof}
  We may think of both matroids as being multisymmetric on $\widetilde E \setminus j$.
  Geometric subsets of $\widetilde E \setminus j$ are of the form $(\cup _{k \in A} E_k) \setminus j$, with $A \subset E$.
  By \cref{determinedbygeo} and \cref{rmk:liftclosure}, it suffices to show that
  \[
    \rk_{\widetilde P \setminus j}((\cup_{k \in A} E_k) \setminus j) = \rk_{R_i P}(A)
  \]
  for all $A \subset E$.
  If $i \not \in A$, then equality plainly holds, so we henceforth assume $i \in A$.
  
  Suppose that $\rk_P(A) = \rk_P(A \setminus i) + \rk_P(i)$.
  In this case, $\rk_{\widetilde P}(A)(\cup_{k \in A} E_k) = \rk_{\widetilde P}(\cup_{k \in A \setminus i} E_k) + \rk_{\widetilde P}(E_i)$, so every element of $E_i$ is a coloop of $\widetilde P |_{\cup_{k \in A} E_k}$.
    Hence,
  \[
    \rk_{\widetilde P \setminus j}((\cup_{k \in A} E_k) \setminus j) = \rk_{\widetilde P}(\cup_{k \in A} E_k) - 1 = \rk_P(A) - 1 = \rk_{R_i P}(A).
  \]
  
  Otherwise, suppose that $\rk_P(A) < \rk_P(A \setminus i) + \rk_P(i)$.
  In this case, \cref{closure} implies that $\overline{(\cup_{k \in A} E_k) \setminus j} \supset \cup_{k \in A} E_k$ in $\widetilde P$.
  Hence,
  \[
    \rk_{\widetilde P \setminus j}((\cup_{k \in A} E_k) \setminus j)
    = \rk_{\widetilde P}((\cup_{k \in A} E_k) \setminus j)
    = \rk_{\widetilde P}((\cup_{k \in A} E_k))
    = \rk_P(A) = \rk_{R_i P}(A).
  \qedhere \]
\end{proof}
We defer discussion of realizing reductions to \cref{rmk:realizingreduction}, where we will see it in action en route to proving \cref{combo}(ii).
\section{Combinatorial flats}\label{comboflats}
We introduce the poset of combinatorial flats of a polymatroid $P$.
We construct the poset using a chosen lift $\widetilde P$ in \cref{comboflatsminimum}, then record the effects of deletion and truncation in \cref{comboflatsdeletiontruncation}.
In Sections \ref{posetprops} and \ref{comboflatsimplification}, we establish order-theoretic properties and consider the effects of reduction, culminating in a proof of \cref{combo}.
Finally, in \cref{comboflataxioms}, we axiomatize posets of combinatorial flats. This provides a cryptomorphic definition of simple polymatroids that generalizes the cryptomorphism of simple matroids and geometric lattices.

Our results on Schubert varieties of subspace arrangements depend only on Sections \ref{comboflatsminimum} and \ref{comboflatsdeletiontruncation}, so geometrically-inclined readers may go directly to \cref{sec:polymatroidschubert} after reading these.

\medskip
Throughout \cref{comboflats}, we fix notations:
\begin{itemize}
\item $(P, \mathbf n)$ is a caged polymatroid on $E = \{1, \ldots, N\}$ with rank function $\rk_P$,
\item $\widetilde P$ is the multisymmetric lift of $(P, \mathbf n)$, on ground set $\widetilde E = E_1 \sqcup \cdots E_N$, and
\item $\rk_P(\mathbf s)$ is the rank of a multiset $\mathbf s$, in the sense of \cref{intropolymatroid}.
\end{itemize}

\subsection{} \label{comboflatsminimum} We construct the poset of combinatorial flats and note some properties of its rank function.
\begin{dfn}\label{defnCombFlat}
    A \word{combinatorial flat} of $(P, \mathbf n)$ is a multiset $\mathbf s \in \N^E$ such that there exists a flat $F$ of $\widetilde P$ satisfying $\mathbf s = (|F \cap E_1|, \ldots, |F \cap E_N|)$.
    (Rephrased in the language of \cref{sec:notation}, $\mathbf s$ is a combinatorial flat if it is represented by a flat of $\widetilde P$.)
\end{dfn}

Ordered by inclusion, the combinatorial flats of $(P, \mathbf n)$ form the \word{poset of combinatorial flats}, denoted $\mathcal L_{P, \mathbf n}$.
Observe that $\mathcal L_{P, \mathbf n}$ is isomorphic to the quotient poset $\mathcal L_{\widetilde P} / \mathfrak S$, in which $\mathfrak S \cdot F \geq \mathfrak S \cdot F'$ if and only if there is $\sigma \in \mathfrak S$ such that $\sigma \cdot F \supset F'$.
Since the action of $\mathfrak S$ on $\mathcal L_{\widetilde P}$ is rank-preserving, $\mathcal L_{P, \mathbf n}$ is a graded poset,
and the rank of $\mathbf s \in \mathcal L_{P, \mathbf n}$ is equal to $\rk_{\widetilde P}(S)$, for any $S \in \mathcal L_{\widetilde P}$ representing $\mathbf s$.

More generally, if $\mathbf s \leq \mathbf n$ is a multiset represented by $S \subset \widetilde E$, then we may define its rank by $\rk_{P, \mathbf n}(\mathbf s) := \rk_{\widetilde P}(S)$.
The \word{closure} of $\mathbf s$ is the combinatorial flat $\overline{\mathbf s} = (\overline s_1, \ldots, \overline s_N)$ represented by $\overline S$, and satisfies $\rk_{P, \mathbf n}(\mathbf s) = \rk_{P, \mathbf n}(\overline{\mathbf s})$. The \word{geometric part} of $\mathbf s$ is $\mathbf s^\geo$, the multiset represented by $S^\geo$.
Multisymmetry of $\widetilde P$ guarantees these definitions do not depend on choice of $S$.

Viewing $\mathcal L_{P, \mathbf n}$ as a quotient and applying \cref{rankformula}, we learn that rank and closure of multisets have the following properties:
\begin{lem}\label{combrankformula}\hfill
  \begin{itemize}
  \item If $\mathbf s$ is a combinatorial flat of $(P, \mathbf n)$, then $\rk_{P,\mathbf n}(\mathbf s) = \rk_{\mathbf n}(\mathbf s^\geo) + |\mathbf s -\mathbf s^\geo|$.
  \item If $\mathbf a \leq \mathbf n$ is a multiset, then for each $1 \leq i \leq N$, either $\overline{a}_i = a_i$ or $\overline{a}_i = n_i$.
  \item If $\mathbf s$ is a combinatorial flat of $(P, \mathbf n)$ and $0 < s_i < n_i$, then $\mathbf s - \mathbf e_i$ is also a combinatorial flat.
  \end{itemize}
\end{lem}
From submodularity of $\rk_{\widetilde P}$, we also obtain:
\begin{lem}\label{stupidbound}
  If $\mathbf s \leq \mathbf n - \mathbf e_i$ is a multiset, then 
  $\rk_{P, n}(\mathbf s + \mathbf e_i) \leq \rk_{P, \mathbf n}(\mathbf s) + 1$.
\end{lem}

A \word{basis} of a multiset $\mathbf s$ is an independent multiset of $P$ maximal among those contained in $\mathbf s$.
\begin{lem}\label{indeprank}
  If $\mathbf b$ is a basis of a multiset $\mathbf s$, then $\rk_{P, \mathbf n}(\mathbf s) = |\mathbf b|$.
  If $\mathbf s$ is a combinatorial flat, then the bases of $\mathbf s$ are precisely multisets of the form $\mathbf b' + \mathbf s - \mathbf s^\geo$, where $\mathbf b'$ is a basis of $\mathbf s^\geo$.
\end{lem}
\begin{proof}
  Choose $S \subset \widetilde E$ representing $\mathbf s$.
  If $\mathbf b$ is a basis of $\mathbf s$,
  then $\mathbf b$ is represented by a basis $B$ of $S$ in $\widetilde P$ by \cref{liftbases}.
  The first assertion holds because $\rk_{P, \mathbf n}(\mathbf s) = \rk_{\widetilde P}(S) = |B| = |\mathbf b|$.

  If $\mathbf s$ is a combinatorial flat, then $S \setminus S^\geo$ consists of coloops of the flat $S$ by \cref{rankformula}.
  Hence, bases of $S$ are of the form $B' \cup (S \setminus S^\geo)$ with $B'$ a basis for $S^\geo$.
  The second assertion follows.
\end{proof}

\begin{rmk}\label{welldfndrk}
  \cref{indeprank} shows that all bases of a multiset $\mathbf s$ have the same cardinality, so $\rk_P(\mathbf s)$, as defined in \cref{intropolymatroid}, is well-defined.
\end{rmk}
\begin{rmk}[Equivalence of definitions]\label{equivalence}
  \cref{indeprank} explains the equivalence of the different definitions of ``rank'' and ``combinatorial flat'' given in \cref{intropolymatroid} and \cref{comboflatsminimum}.
  The first statement of \cref{indeprank} says exactly that $\rk_{P, \mathbf n}(\mathbf s) = \rk_P(\mathbf s)$ for any multiset $\mathbf s$. When $\mathbf s \in \mathcal L_{P, \mathbf n}$, this means $\rk_P(\mathbf s)$ is also equal to the rank of $\mathbf s$ in the graded poset $\mathcal L_{P, \mathbf n}$.
  We will henceforth write only $\rk_P(\mathbf s)$.

  For equivalence of the two definitions of ``combinatorial flat'', choose $S \subset \widetilde E$ representing $\mathbf s$.
  If $s_i < n_i$, then $\mathbf s + \mathbf e_i$ is represented $S \cup e$ for any $e \in E_i \setminus S$.
  In the sense of \cref{intropolymatroid}, $\mathbf s$ is a combinatorial flat if and only if
  \[
    \rk_{\widetilde P}(S \cup e) = \rk_{P, \mathbf n}(\mathbf s + \mathbf e_i) = \rk_P(\mathbf s + \mathbf e_i)  > \rk_P(\mathbf s) = \rk_{P, \mathbf n}(\mathbf s) = \rk_{\widetilde P}(S)
  \]
  for all $e \in E_i$ with $s_i < n_i$.
  This is precisely the statement that $S$ is a flat of $\widetilde P$, which is the notion of ``combinatorial flat'' defined in \cref{comboflatsminimum}.
\end{rmk}
\begin{rmk}[Geometry of combinatorial flats]\label{geometryofcf}
  Suppose that the polymatroid $P$ is realized by a subspace arrangement $V_1, \ldots,~V_N~\subset~V$.
  Rephrasing \cref{realizablelift} in the language of arrangements, if $\mathbf n$ is a cage for $P$, then $\widetilde P$ is realized by any arrangement $\{H_{ij} : 1 \leq i \leq N, 1 \leq j \leq n_i\}$ in which $H_{i1},\ldots, H_{in_i}$ are general hyperplanes containing $V_i$.
  The lattice of flats $\mathcal L_{\widetilde P}$ has an action by a product of symmetric groups $\mathfrak S$, and the quotient $\mathcal L_{\widetilde P} / \mathfrak S$ is isomorphic to $\mathcal L_P$ (see also \cref{comboflatsminimum}).

  Alternatively, consider the subspace arrangement $\{V_{ij} : 1 \leq i \leq N, 1 \leq j \leq \codim V_i\}$, where $V_i = V_{i,\codim_V V_i} \subset \cdots \subset V_{i,2} \subset V_{i,1} \subsetneq V$ is a general flag.
  The poset of flats (\emph{not} combinatorial flats) of the polymatroid associated to this arrangement is isomorphic to $\mathcal L_P$. If $P$ is simple and $\mathbf n = (\rk_P(1), \ldots, \rk_P(N))$, then the rank of $\mathbf s \in \mathcal L_{P, \mathbf n}$ is $\rk_P(\mathbf s) = \codim_V V_{1,s_1} \cap \cdots \cap V_{N, s_N}$.
\end{rmk}
\begin{rmk}
  \label{matricube}
  \word{Matricubes} were recently introduced by \cite{AG24} to model intersection patterns of flags of linear subspaces.
  If $(P, \mathbf n)$ is a caged polymatroid, then the function $\rk_P$ on multisets contained in $\mathbf n$ defines a matricube.
  The \word{flats} \cite[Definition 3.1]{AG24} of this matricube are the combinatorial flats of $P$, and the \word{independents} \cite[Definition 5.1]{AG24} of the matricube are the independent multisets of $P$.

  From the flag perspective of \cref{geometryofcf}, the difference between matricubes and combinatorial flats is that matricubes model \emph{all} collections of flags, while combinatorial flats model only collections of \emph{general} flags.
  Combinatorially, this difference manifests in several ways: combinatorial flats obey \cref{combrankformula}, are top-heavy, and admit a reasonable notion of basis, but none of these properties hold for matricubes in general. (Failure of top-heaviness and \cref{combrankformula} for matricubes can be seen in the right-hand example of \cite[Section 3.1]{AG24}. The difficulties of bases for matricubes are discussed in \cite[Section 9.1]{AG24}.)
\end{rmk}

\subsection{}\label{comboflatsdeletiontruncation}
We record the effects deletion and truncation on $\mathcal L_{P, \mathbf n}$.
\begin{lem}\label{localproduct}
  Let $i \in E$. If $\mathbf s \leq \mathbf n$ is a multiset with $\overline s_i < n_i$, then $\pi_{E \setminus i}(\mathbf s)$ is a combinatorial flat of $(P \setminus i, \pi_{E \setminus i}(\mathbf n))$ if and only if $\mathbf s$ is a combinatorial flat of $(P, \mathbf n)$.
\end{lem}
\begin{proof}
  \revision{The if direction follows from \cref{deletionflats} and \cref{deletioncommutes}}.
  For the ``only if'', suppose that $\pi_{E \setminus i}(\mathbf s)$ is a combinatorial flat of $(P \setminus i, \pi_{E \setminus i}(\mathbf n))$. We will show that $\mathbf s$ must be a flat by proving $\mathbf s = \overline{\mathbf s}$.
  By \cref{combrankformula} and the hypothesis that $\overline s_i < n_i$, it is always the case that $s_i = \overline s_i$.
  Hence, it suffices to show that $s_j = \overline s_j$ for $j \neq i$.

  \revision{By \cref{deletionflats} and \cref{deletioncommutes}}, there is a combinatorial flat $\mathbf s'$ of $(P, \mathbf n)$ with $\pi_{E \setminus i}(\mathbf s') = \pi_{E \setminus i}(\mathbf s)$.
  If $s'_i = n_i$, then $\mathbf s < \mathbf s'$, so $\mathbf s \leq \overline{\mathbf s} < \mathbf s'$.
  Entrywise, this implies $s_j \leq \overline s_j \leq s_j' = s_j$, which gives the desired equality.
  
  Otherwise,  $s_i' < n_i$.
  By Lemmas \ref{combrankformula} and \ref{stupidbound}, $\mathbf s'' := \mathbf s' - s_i' \mathbf e_i = \mathbf s - s_i \mathbf e_i$ is a combinatorial flat of rank at least $\rk(\mathbf s) - s_i$.
  By \cref{combrankformula}, $\overline{\mathbf s} - s_i \mathbf e_i$ is also a combinatorial flat, and has rank $\rk(\overline{\mathbf s}) - s_i$.
  Since $\mathbf s'' \leq \overline{\mathbf s} - s_i \mathbf e_i$, these two combinatorial flats must be equal.
  In particular, for $j \neq i$, $s_j'' = \overline s_j$. Since $s_j'' = s_j$ by construction, this means $s_j =\overline s_j$, as desired.
\end{proof}

\begin{lem}\label{truncation}
  Let $(P, \mathbf n)$ be a caged polymatroid on $E$.
  Let $F \subset E$ be a flat, and $\mathbf f = \sum_{i \in F} n_i \mathbf e_i$.
  The combinatorial flats of $(T_F P, \mathbf n)$ are the combinatorial flats $\mathbf s$ of $P$ that either contain $\mathbf f$ or satisfy  $\overline{\mathbf s +  \mathbf e_i} \not \geq \mathbf f$ for every $i \in E$ such that $s_i < n_i$.
  In the former case, $\rk_{T_FP}(\mathbf s) = \rk_P(\mathbf s) - 1$, and in the latter, $\rk_{T_FP}(\mathbf s) = \rk_P(\mathbf s)$.
\end{lem}
\begin{proof}
  For matroids and ordinary flats, this lemma is proved in \cite[Proposition 7.4.9]{B86}. Explicitly, if $P$ is a matroid, then the flats of $T_F P$ are the flats $S$ of $P$ that either contain $F$ or satisfy $\overline{S \cup i} \not \supset F$ for every $i \in E \setminus S$.
  
  By the matroid case and \cref{truncationcommutes}, a set $S$ is a flat of $\widetilde{T_F P}$ if and only if $S$ is a flat of $\widetilde P$ that either contains $\cup_{i \in F} E_i$ or satisfies $\overline{S \cup i} \not \supset \cup_{i \in F} E_i$ for all $i \in \widetilde E \setminus S$. In the former case, $\rk_{\widetilde{T_F P}}(S) = \rk_{\widetilde P}(S) - 1$; in the latter $\rk_{\widetilde{T_FP}}(S) = \rk_{\widetilde P}(S)$.
Taking $\mathbf s$ to be the multiset represented by $S$ completes the proof.
\end{proof}

\bigskip
In the remainder of \cref{comboflats}, we work up to proving \cref{combo}, then provide axioms for combinatorial flats.
These results are not required to understand Schubert varieties of subspace arrangements; readers interested solely in geometry may proceed to  \cref{sec:polymatroidschubert} forthwith.

\subsection{}\label{posetprops}
We prove that $\mathcal L_{P, \mathbf n}$ is a top-heavy semimodular lattice, and identify its join-irreducibles in a special case.
For multisets $\mathbf s, \mathbf s' \leq \mathbf n$, let  $\mathbf s \vee \mathbf s' := \overline{(\max(s_i, s_i'))_i}$ and $\mathbf s \wedge \mathbf s' := (\min(s_i, s_i'))_i$.

\begin{lem}\label{submodularity}
  If $\mathbf s$ and $\mathbf s'$ are multisets contained in $\mathbf n$, then
  \[
    \rk_{P}(\mathbf s \wedge \mathbf s') + \rk_{P}(\mathbf s \vee \mathbf s') \leq \rk_{P}(\mathbf s) + \rk_{P}(\mathbf s').
  \]
\end{lem}
\begin{proof}
  Pick sets $S, S' \subset \widetilde E$ representing $\mathbf s$ and $\mathbf s'$, respectively, such that $S \cap S'$ represents $\mathbf s \wedge \mathbf s'$ and $S \cup S'$ represents $(\max(s_i, s_i'))_i$.
  By submodularity of $\rk_{\widetilde P}$,
  \begin{align*}
    \rk_{P}(\mathbf s \wedge \mathbf s') + \rk_{P}(\mathbf s \vee \mathbf s')
    &= \rk_{\widetilde P}(S \cap S') + \rk_{\widetilde P}(\overline{S \cup S'})
    = \rk_{\widetilde P}(S \cap S') + \rk_{\widetilde P}(S \cup S')\\
    &\leq \rk_{\widetilde P}(S) + \rk_{\widetilde P}(S')
    = \rk_{P}(\mathbf s) + \rk_{P}(\mathbf s'). \qedhere
  \end{align*}
\end{proof}

\begin{lem}\label{lattice}
  $\mathcal L_{P, \mathbf n}$ is a semimodular lattice, with join and meet given by $\vee$ and $\wedge$, respectively. Moreover, $\mathcal L_{P, \mathbf n}$ is top heavy and $|\mathcal L_{P, \mathbf{n}}^k| \leq | \mathcal L_{P, \mathbf n}^{k+1}|$ when $k \leq d/2$.
\end{lem}
\begin{proof}
  Let $\mathbf s, \mathbf s' \in \mathcal L_{P, \mathbf n}$.
  The join of $\mathbf s$ and $\mathbf s'$ is equal to $\mathbf s \vee \mathbf s'$ because any combinatorial flat containing $\mathbf s$ and $\mathbf s'$ must contain $(\max(s_i, s_i'))_i$, hence must also contain $\mathbf s \vee \mathbf s'$.

  For meets: any combinatorial flat contained in both $\mathbf s$ and $\mathbf s'$ is contained in $\mathbf s \wedge \mathbf s'$, so we need only show that $\mathbf s \wedge \mathbf s'$ is a combinatorial flat.
  To see this, pick sets $S, S' \subset \widetilde E$ representing $\mathbf s$ and $\mathbf s'$, and such that $S \cap S'$ represents $\mathbf s \wedge \mathbf s'$. Since $\mathbf s$ and $\mathbf s'$ are combinatorial flats, $S$ and $S'$ are flats of $\widetilde P$.
  Since flats are closed under intersection, $\mathbf s \wedge \mathbf s'$ is also a combinatorial flat.
  Hence, $\mathbf s \wedge \mathbf s'$ is the meet of $\mathbf s$ and $\mathbf s'$ in $\mathcal L_{P, \mathbf n}$.

  Semimodularity is \cref{submodularity}.

  For top-heaviness:
  let $H^k(\widetilde P)$ be the $\Q$-vector space basis $\{v_F : F \in \mathcal L_{\widetilde P}^k\}$, and let
  $H^k(P, \mathbf n)$ be the $\Q$-vector space with basis $\{v_{\mathbf s} : \mathbf s \in \mathcal L_{P, \mathbf n}^k\}$.
  From the quotient construction of $\mathcal L_{P, \mathbf n}$, there are isomorphisms
  \[
    H^k(P, \mathbf n) \overset{\cong}{\to} H^k(\widetilde P)^{\mathfrak S} \subset H^k(\widetilde P), \quad v_{\mathbf s} \mapsto \sum_{\sigma \in \mathfrak S} v_{\sigma \cdot S},
  \]
  where $S$ is any flat representing $\mathbf s$.
  By \cite[Theorem 1.1]{BHMPW20b}, there is an injective linear $L: H^k(\widetilde P) \to H^{d-k}(\widetilde P)$ that commutes with the action of $\mathfrak S$. Consequently, the restriction of $L$ to $H^k(\widetilde P)^{\mathfrak S}$ gives a linear injection
  \[
    H^k(P, \mathbf n) \cong H^k(\widetilde P)^{\mathfrak S} \lhook\joinrel\xrightarrow{\;L\;} H^{d-k}(\widetilde P)^{\mathfrak S} \cong H^{d-k}(P, \mathbf n),
  \]
  and taking dimensions yields top-heaviness.
  Also by \cite[Theorem 1.1]{BHMPW20b}, there are $\mathfrak S$-equivariant linear injections $H^0(\widetilde P) \hookrightarrow H^{1}(\widetilde P) \hookrightarrow \ldots \hookrightarrow H^{\lfloor d/2 \rfloor}(\widetilde P)$, which restrict to linear injections
  \[
    H^k(P, \mathbf n) \cong H^k(\widetilde P)^{\mathfrak S} \lhook\joinrel\xrightarrow{\;L\;} H^{k+1}(\widetilde P)^{\mathfrak S} \cong H^{k+1}(P, \mathbf n),
  \]
  which proves the remaining statement.
\end{proof}
Generalizations of Lemmas \ref{submodularity} and \ref{lattice} for matricubes appear in \cite[Theorem 3.4]{AG24} (see \cref{matricube}).
\revision{With Lemmas \ref{submodularity} and \ref{lattice} in hand, we have established \cref{combo} parts (i), (iii), and (iv). We now devote our attention to \cref{combo}(ii), which states that $P$ is determined by $\mathcal L_P$ up to simplification. }

\revision{An element of a lattice is \word{join-irreducible} if it cannot be written as a join of strictly smaller elements.}
\begin{lem}\label{joinirreds}
  If $P$ is simple and $\mathbf n = (\rk_P(1), \ldots, \rk_P(N))$, then the join-irreducibles of $\mathcal L_{P, \mathbf n}$ are all multisets $s_i \mathbf e_i$, with $i \in E$ and $1 \leq s_i \leq n_i$.
\end{lem}
\begin{proof}
  Since $P$ is simple, $E_i$ is a flat of $\widetilde P$ for all $i \in E$.
  The restriction of $\widetilde P$ to $E_i$ is a boolean matroid of rank $n_i = \rk_P(i)$, so all proposed join-irreducibles are combinatorial flats of $(P, \mathbf n)$.
  Each $s_i \mathbf e_i$ is join-irreducible because the interval $[\mathbf 0, s_i \mathbf e_i] \subset \mathcal L_{P, \mathbf n}$ is a chain of length $s_i$.
  No other combinatorial flats are join-irreducible because $\mathbf s \in \mathcal L_{P, \mathbf n}$ can be written $\mathbf s = s_1 \mathbf e_1 \vee \cdots \vee s_N \mathbf e_N$.
\end{proof}
\begin{cor}\label{determined}
  If $P$ is simple and $\mathbf n = (\rk_P(1), \ldots, \rk_P(N))$, then $P$ is determined up to relabelling of the ground set by the poset $\mathcal L_{P, \mathbf n}$.
\end{cor}
\begin{proof}
  By \cref{joinirreds}, each $i \in E$ corresponds to a maximal join-irreducible $x_i$ of $\mathcal L_{P, \mathbf n}$.
  The rank function of $P$ is recovered by $\rk_P(S) = \rk_{P, \mathbf n}(\vee_{i \in S} x_i)$.
\end{proof}

\subsection{} \label{comboflatsimplification}
We show that reduction and loops do not change the lattice of combinatorial flats.
From these facts, we derive \cref{combo}.

\begin{lem}\label{reductioniso}
  Suppose $P$ is loopless.
If  $\{i\} \subset E$ is not a flat of $P$ with rank $n_i$, then $\mathcal L_{R_i(P, \mathbf n)} \cong \mathcal L_{P, \mathbf n}$.
\end{lem}
\begin{proof}
  We will show the map $\psi: \mathcal L_{P, \mathbf n} \to \mathcal L_{R_i(P, \mathbf n)}$ defined by
  \[ \psi(\mathbf s) = \begin{cases}
    \mathbf s, & s_i < n_i \\
    \mathbf s - \mathbf e_i, & s_i = n_i
  \end{cases}
\]
is an isomorphism.
Multisets in the image of $\psi$ are combinatorial flats of $R_i(P, \mathbf n)$ by Lemmas \ref{reductionlift} and \ref{deletionflats}.

For surjectivity, suppose that $\mathbf s \in \mathcal L_{R_i(P, \mathbf n)}$ is represented by a flat $S$ of $\widetilde{R_i(P, \mathbf n)}$.
We will show there is $\mathbf s' \in \mathcal L_{P, \mathbf n}$ with $\psi(\mathbf s') = \mathbf s$.
By Lemmas \ref{reductionlift} and \ref{deletionflats},  there is $e \in E_i \setminus S$ such that $S \cup e$ or $S$ is a flat of $\widetilde P$.
If $S$ is a flat of $\widetilde P$, then $s_i < n_i$, so we may take $\mathbf s' = \mathbf s$.
Otherwise, suppose $S \cup e$ is a flat of $\widetilde P$.
If $s_i + 1 < n_i$, then by \cref{combrankformula}, $\mathbf s$ is a combinatorial flat of $\widetilde P$, so we may again take $\mathbf s' = \mathbf s$.
Otherwise, if $S \cup e$ is a flat of $\widetilde P$ and $s_i + 1 = n_i$, then take $\mathbf s' = \mathbf s + \mathbf e_i$.
This completes the proof of surjectivity.

We now prove injectivity.
Let $\mathbf s \in \mathcal L_{R_i(P, \mathbf n)}$.
It is only possible for $\mathbf s$ to have more than one preimage if $s_i = n_i - 1$.
In this case, the two possible preimages of $\mathbf s$ are $\mathbf s$ and $\mathbf s + \mathbf e_i$.

If $\rk_P(i) < n_i$, then $\mathbf s$ cannot be a combinatorial flat of $(P, \mathbf n)$.
This is because, assuming the contrary, we obtain the false inequality
\begin{align*}
  \rk(\mathbf s) < \rk(\mathbf s + \mathbf e_i)
  &= \rk((\mathbf s - (n_i-1) \mathbf e_i) \vee n_i \mathbf e_i) \\
  &\leq \rk(\mathbf s - (n_i-1) \mathbf e_i) + \rk(n_i \mathbf e_i) \quad \text{by \cref{submodularity}}\\
  &= \rk(\mathbf s - (n_i-1) \mathbf e_i) + \rk_P(i) \\
  &\leq \rk(\mathbf s - (n_i-1) \mathbf e_i) + n_i - 1 \\
  & = \rk(\mathbf s) \quad \text{by \cref{combrankformula} and the contrarian assumption.}
\end{align*}
Since $\mathbf s$ is not a combinatorial flat of $(P, \mathbf n)$, $\mathbf s$ has only one preimage when $\rk_P(i) < n_i$.

Otherwise, we are in the case where $\rk_P(i) = n_i$ and $\{i\}$ is not a flat of $P$.
Choose sets $S \subset S' \subset \widetilde E$ representing $\mathbf s$ and $\mathbf s + \mathbf e_i$, respectively, and
assume towards a contradiction that $S$ and $S'$ are both flats of $\widetilde P$.
Let $e'$ be an element of $\overline{E_i} \setminus E_i$, which is nonempty because $\{i\}$ is not a flat of $P$.
If $e$ is the unique element of $S' \setminus S$, then $S$ contains $\overline{(E_i \setminus e) \cup e'}$.

We now show that $\overline{(E_i \setminus e) \cup e'} = \overline{E_i}$, which implies that $S = S'$, a contradiction.
By the definition of $\widetilde P$, there is $A \subset E$ such that
\[
  \rk_{\widetilde P}((E_i \setminus e) \cup e') = \rk_P(A) + |((E_i \setminus e) \cup e') \setminus \cup_{k \in A} E_k|.
\]
If $i \in A$, then $\rk_{\widetilde P}((E_i \setminus e) \cup e') \geq \rk_P(i)$, so $\overline{(E_i \setminus e) \cup e'} = \overline{E_i}$.
Otherwise, if $i \not \in A$, then
\[ \rk_{\widetilde P}((E_i \setminus e) \cup e') = \rk_P(A) + |E_i \setminus e| + |\{e'\} \setminus \cup_{k \in A} E_k| = \rk_P(A) + \rk_P(i) - 1 + |\{e'\} \setminus \cup_{k \in A} E_k|. \]
This quantity is always at least $\rk_P(i)$ because $P$ is loopless; therefore, $\overline{(E_i\setminus e) \cup e'} = \overline{E_i}$.
\end{proof}
\begin{lem}\label{deloop}
  If $i$ is a loop of $P$, then $ \mathcal L_{P, \mathbf n} \cong \mathcal L_{P\setminus i, \pi_{E \setminus i}(\mathbf n)}$.
\end{lem}
\begin{proof}
  If $i$ is a loop of $P$, every element of $E_i$ is a loop of $\widetilde P$, and the $\mathfrak S_{E_i}$ factor of $\mathfrak S$ acts trivially on $\mathcal L_{\widetilde P}$.
  Hence,
  $\mathcal L_{P, \mathbf n} \cong \mathcal L_{\widetilde P} / \mathfrak S  \cong \mathcal L_{\widetilde{P \setminus i}} /  \pi_{E \setminus i}(\mathfrak S) \cong \mathcal L_{P\setminus i, \pi_{E \setminus i}(\mathbf n)}$.
\end{proof}

\begin{proof}[Proof of \cref{combo}]
  Define $\mathcal L_P := \mathcal L_{P, \mathbf n}$, with $\mathbf n$ any cage for $P$.
  This definition does not depend on the choice of cage because we may assume that $P$ is loopless by \cref{deloop}, then repeatedly apply \cref{reductioniso} to produce an isomorphism $\mathcal L_{P, \mathbf n} \cong \mathcal L_{P, (\rk_P(1), \ldots, \rk_P(N))}$.
  We will henceforth assume that $\mathbf n = (\rk_P(1), \ldots, \rk_P(N))$.

  Statements \cref{combo:semimodularlattice} and \cref{combo:topheavy} are the content of \cref{lattice}; we now move to proving \cref{combo:simple}.
  To construct $P^\simp$: let $(P', \mathbf n') = R_i(P \setminus L, \pi_{E \setminus L}(\mathbf n))$, where $L$ is the set of loops of $P$ and $i \in E \setminus L$ is any element such that $\{i\}$ is not a flat.
  If the resulting polymatroid $P'$ is simple, then take $P^\simp := P'$; otherwise, repeat these steps, replacing $(P,\mathbf n)$ with $(P', \mathbf n')$.
  This process terminates because the cage shrinks at each step.
  If the cage is $\mathbf 0$, then $P$ is the polymatroid of rank 0 on zero elements, which is simple; otherwise, the process stopped before the cage was equal to $\mathbf 0$, meaning that a simple polymatroid was produced.

  Lemmas \ref{reductioniso} and \ref{deloop} show that $\mathcal L_{P^\simp} \cong \mathcal L_P$, and \cref{determined} implies $P^\simp$ is the unique simple polymatroid with this property.
 This completes the proof of \cref{combo}.
\end{proof}
\begin{rmk}\label{rmk:realizingreduction}
  If $P$ is loopless and realized by a subspace arrangement $V_1, \ldots, V_N \subset V$, then
  $P$ fails to be simple if and only if $V_i \subset V_j \subsetneq V$  for some $i \neq j$.
  In this case, a step of the ``simplification process'' used to construct $P^\simp$ in the proof of \cref{combo} corresponds to replacing  $V_i$ with $V_i + \ell$, where $\ell \subset V$ is a generic line, then deleting loops.
  Alternatively, one may view this as replacing $V_i$ with its immediate successor $V_{i,\codim V_i - 1}$ in a general flag containing $V_i$ (see \cref{geometryofcf} and \cref{zb:simplification}).
\end{rmk}
\begin{zb}\label{zb:simplification}
  Suppose that $P$ is realized by the subspace arrangement $V_1 = V_2 = 0 \subseteq V = \C^2$.
  Complete these subspaces to general flags $0 = V_1 \subsetneq V_{1,1} \subsetneq V$ and $0 = V_2 \subsetneq V_{2,1} \subsetneq V$, and let $\mathbf{n} = (2,2)$. The steps of the reduction process for $(P,\mathbf n)$ are illustrated below. 
    
\begin{center}    
      \centering
      \begin{minipage}{0.3\textwidth}
\centering
\begin{tikzpicture}
    \node (first) at (0,1.6) {$\mathcal{L}_{P,\mathbf{n}} = \{\mathbf{0}, \mathbf{e_1}, \mathbf{e_2}, 2\mathbf{e_1} + 2\mathbf{e_2}\}$} ;
    \filldraw (0,0) circle (2pt) node[left=10pt] {$V_1=V_2$};
    
    \draw[dotted, thick] (-2,-1) -- (2,1) node[near end, above] {$V_{1,1}$};
    \draw[dotted, thick] (-2,1) -- (2,-1) node[near end, below] {$V_{2,1}$};

\end{tikzpicture}
\end{minipage}%
\quad
\begin{minipage}{0.3\textwidth}
\centering
\begin{tikzpicture}
     \node (first) at (0,1.6) {$\mathcal{L}_{R_1(P,\mathbf{n})} = \{\mathbf{0}, \mathbf{e_1}, \mathbf{e_2}, \mathbf{e_1} + 2\mathbf{e_2}\}$} ;
    \filldraw (0,0) circle (2pt) node[left=10pt] {$V_2$};
    
    \draw[ thick] (-2,-1) -- (2,1) node[near end, above] {$V_1$};
    \draw[dotted, thick] (-2,1) -- (2,-1) node[near end, below] {$V_{2,1}$};

\end{tikzpicture}
\end{minipage}%
\quad
\begin{minipage}{0.3\textwidth}
\centering
\begin{tikzpicture}
       \node (first) at (0,1.6) {$\mathcal{L}_{R_2(R_1(P,\mathbf{n}))} = \{\mathbf{0}, \mathbf{e_1}, \mathbf{e_2}, \mathbf{e_1} + \mathbf{e_2}\}$} ;
    \filldraw (0,0) circle (2pt);
    
    \draw[ thick] (-2,-1) -- (2,1) node[near end, above] {$V_1$};
    \draw[ thick] (-2,1) -- (2,-1) node[near end, below] {$V_2$};

\end{tikzpicture}
\end{minipage}

\end{center}
\end{zb}
\begin{rmk}
  If $j \in E$ is such that $\{j\}$ is a flat of $P$ with positive rank, then $\{j\}$ is also a flat of $P^\simp$.
  This fact follows from the claim below, which also provides a concrete way to see that the simplification process terminates.
    \begin{claim}
    If $\{j\}$ is a flat of a loopless polymatroid $P$ and $i \in E \setminus j$, then either $\{i\}$ and $\{j\}$ are both flats of $R_i P$, or $i$ is a loop of $R_i P$ and $\{i,j\}$ is a flat of $R_i P$.
  \end{claim}
  \begin{claimproof}
    Suppose $i$ is a loop of $R_i P$.
    If $k \in E \setminus \{i,j\}$, then
    \[ \rk_{R_i P}(\{i,j,k\}) = \rk_{R_i P}(\{j,k\}) = \rk_P(\{j,k\}) > \rk_P(j) = \rk_{R_i P}(\{i,j\}), \]
    where the second equality holds because $P$ is loopless. Hence, $\{i,j\}$ is a flat of $R_i P$.

    Otherwise, suppose $i$ is not a loop of $R_i P$. In this case, $\rk_P(i) \geq 2$.
    If $k \in E \setminus i$ lies in the closure of $\{i\}$ in $P$, then
    \[
      \rk_{R_i P}(\{i,k\}) = \rk_P(\{i,k\}) = \rk_P(i) > \rk_P(i) - 1 = \rk_{R_i P}(i).
    \]
    If $k$ does not lie in the closure of $\{i\}$ in $P$, then
    \[
      \rk_{R_i P}(\{i,k\}) \geq \rk_P(\{i,k\}) - 1 > \rk_P(i) - 1 = \rk_{R_i P}(i).
    \]
    This shows that $\{i\}$ is a flat of $R_i P$. One similarly checks that $\{j\}$ remains a flat in $R_i P$.
    If $k \in E \setminus \{i,j\}$, then
    \[
      \rk_{R_i P}(\{j,k\}) = \rk_P(\{j,k\}) > \rk_P(j) = \rk_{R_i P}(j).
    \]
    If $\rk_P(\{i,j\}) = \rk_P(j) + \rk_P(i)$, then
    \[
      \rk_{R_i P}(\{i,j\}) = \rk_P(j) + \rk_P(i) - 1 \geq \rk_P(j) + 1 > \rk_P(j) = \rk_{R_i P}(j);
    \]
    otherwise,
    \[
      \rk_{R_i P}(\{i,j\}) = \rk_P(\{i,j\}) > \rk_P(j) = \rk_{R_i P}(j).
    \]
  \end{claimproof}
\end{rmk}

\subsection{}\label{comboflataxioms}
We axiomatically characterize posets of combinatorial flats.
In light of \cref{combo}, this provides a cryptomorphic definition of simple polymatroids, generalizing the equivalence of geometric lattices and simple matroids.

Let $\mathcal L$ be a graded lattice with minimum element $\hat 0$.
The \word{nullity} of $e \in \mathcal L$ is
\[
  \nul(e) = \#\{\text{join-irreducibles below $e$}\} - \rk_{\mathcal L}(e),
\]
where  $\rk_{\mathcal L}(e)$ is the length of any saturated chain in $[\hat 0, e]$.
The \word{geometric part} of $e$, denoted $e^\geo$, is the join of all elements of $[\hat 0, e]$ that are maximal irreducibles of $\mathcal L$.
\begin{thm}\label{latticeaxiom}
  A \revision{graded} poset $\mathcal L$ is the lattice of combinatorial flats of a polymatroid $P$ if and only if
  \begin{itemize}
  \item $\mathcal L$ is a semimodular lattice,
  \item the join-irreducibles of $\mathcal L \setminus \hat 0$ form a downward-closed set,
  \item each join-irreducible of $\mathcal L \setminus \hat 0$ is below a unique maximal join-irreducible, and
  \item for any $e \in \mathcal L$, $\nul(e) = \nul(e^\geo)$.
  \end{itemize}
\end{thm}
\begin{proof}
  If $P$ is a polymatroid (which we may assume simple by \cref{combo}\cref{combo:simple}), then the prescribed properties hold for $\mathcal L_P$ by \cref{combo}\cref{combo:semimodularlattice}, \cref{joinirreds}, and \cref{combrankformula}.

  Conversely, suppose that $\mathcal L$ is a poset with the prescribed properties.
  Let $E$ be the set of maximal join-irreducibles of $L$ and define a simple polymatroid $P$ by $\rk_P(A) := \rk_{\mathcal L}(\vee_{i \in A} i)$.
  Let $\mathbf n = (\rk_{\mathcal L}(i))_{i \in E}$.

  Since the irreducibles of $\mathcal L$ are downward closed, for each $i \in E$,
  $[\hat 0, i]  = \{\hat 0 < x_{i,1} < \cdots < x_{i, n_i} = i\}$.
  By construction of $P$, the following are equivalent:
  \begin{itemize}
  \item $F \subset E$ is a flat of rank $r$
  \item  $\{i : x_{i,n_i} \leq \vee_{j \in F} x_{j, n_j}\} = F$ and $\rk_{\mathcal L}(\vee_{j \in F}  x_{j, n_j}) = r$,
  \item and $\sum_{j \in F} n_j \mathbf e_j$ is a combinatorial flat of rank $r$.
  \end{itemize}
  Define maps
  \begin{align*}
    \ph: \mathcal L_P \cong \mathcal L_{P, \mathbf n} \to \mathcal L,& \quad \ph(\mathbf s) = \vee_{i \in E}\, x_{i, s_i} \\
    \ph': \mathcal L \to \mathcal L_{P, \mathbf n} \cong \mathcal L_P, &\quad x \mapsto \vee_{x_{i,j} \leq x}\, j \mathbf e_i.
  \end{align*}
  Plainly, the compositions $\ph \circ \ph'$ and $\ph' \circ \ph$ are increasing maps \revision{(i.e. $(\ph \circ \ph')(x) \geq x$ and $(\ph' \circ \ph)(\mathbf s) \geq \mathbf s$)}.
  To finish, we show that both $\ph$ and $\ph'$ decrease rank.
  Let $\mathbf s$ be a combinatorial flat and let $F = \{i : s_i = n_i\}$, so that $\mathbf s^\geo = \sum_{i \in F} n_i \mathbf e_i$.
  Using submodularity of $\rk_{\mathcal L}$, our discussion of flats of $P$, and \cref{combrankformula}, we obtain
  \begin{align*}
    \rk_{\mathcal L}(\ph(\mathbf s))
    &\leq \rk_{\mathcal L}(\vee_{i \in F} x_{i, n_i}) + \sum_{i \in E \setminus F} s_j 
    = \rk_P(F) + \sum_{i \in E \setminus F} s_j
    = \rk_P(\mathbf s^\geo) + |\mathbf s - \mathbf s^\geo|
    = \rk_P(\mathbf s).
  \end{align*}
  This shows that $\ph$ decreases rank.
  On the other hand, if $x \in \mathcal L$ and $F = \{i : x_{i, n_i} \leq x\}$, then
  \[ \sum_{i \in E} \max \{j : x_{ij} \leq x\} - \rk_{\mathcal L}(x) = \nul(x) = \nul(x^\geo) = \sum_{i: x_{i, n_i} \leq x} n_i - \rk_{\mathcal L}(x^\geo), \]
  so
  \begin{align*}
    \rk_{\mathcal L}(x) &= \rk_{\mathcal L}(x^\geo) + \sum_{i : x_{i, n_i} \not \leq x} \max\{j : x_{i,j} \leq x\}\\
    &= \rk_P\left(\sum_{i \in F} n_i \mathbf e_i\right) + \sum_{i : x_{i, n_i} \not \leq x} \max\{j : x_{i,j} \leq x\}
    \geq \rk_P( \vee_{x_{i,j} \leq x} j \mathbf e_j) = \rk_P(\ph'(x)).
  \end{align*}
  This shows that $\ph'$ also decreases rank.
  Hence, both $\ph \circ \ph'$ and $\ph' \circ \ph$ are increasing maps that decrease rank, which means they must be the identity.
\end{proof}

\section{Schubert varieties of subspace arrangements}\label{sec:polymatroidschubert}
We introduce and study the Schubert variety $Y_V \subset \P^{\mathbf n}$ associated to a linear subspace $V \subset \K^{\mathbf n}$, with $\K$ an algebraically closed field of characteristic 0.
Our construction is founded on work of Hassett and Tschinkel \cite{HT99}, which we review in Sections \ref{sec_HT} and \ref{rmk_Gm}.
We construct $Y_V$ in \cref{polymatschubertdfn}, then in \cref{chowclass} compute the support of its class in the Chow ring of $\P^{\mathbf n}$.
In \cref{pgprops}, we discuss ``polymatroid generality'', a condition that guarantees $Y_V$ consists of finitely many $V$-orbits.
We assume this condition in \cref{geomproofs}, where we use the results of \cref{chowclass} to identify strata of $Y_V$, ending with a proof of \cref{main:stratification}.
Finally, in \cref{sec:topresults}, we give a presentation for the singular cohomology ring of $Y_V$ when $\K = \C$ and $V$ is polymatroid general.

\medskip

\subsection{}\label{sec_HT}
We explain a construction of Hassett and Tschinkel \cite{HT99}, which we will eventually think of as the variety associated to a polymatroid on one element.
Loosely speaking, Hassett and Tschinkel show that there is only one action of the $n$-dimensional additive group $\G_a^n$ on $\P^n$ with finitely many orbits. More precisely:
\begin{thm}\cite[Proposition 3.7]{HT99}\label{HT}
  Over an algebraically closed field of characteristic 0, there is a left action of $\G_a^n$ on $\P^n$ with finitely many orbits.
  The action is unique up to isomorphism of $\G_a^n$-variety structures.
\end{thm}
To make matters concrete, choose coordinates $a_1, \ldots, a_n$ and $b_0,\ldots, b_n$ on  $\Ga^n$ and $\P^n$, respectively.
The action on $\P^n$ is given by the faithful representation $\rho_n: \Ga^n \to \Aut(\P^n) = \PGL_{n+1}$,
\[
  \mathbf a = (a_1, \ldots, a_n) \mapsto \exp
                                  \begin{pmatrix}
                                    0 & 0 & 0 & \cdots & 0\\
                                    a_1 & 0 & 0 & \ddots & \vdots\\
                                    a_2 & a_1 & 0 & \ddots & 0 \\
                                    \vdots & \ddots & \ddots & \ddots & 0 \\
                                    a_n & \cdots & a_2 & a_1 & 0
                                  \end{pmatrix}.
\]
The entries in each northwest-to-southeast diagonal of $\rho_n(\mathbf a)$ are all equal to one another.
Explicitly, the $(i,j)$ entry is
\[
  \big(\rho_n(\mathbf a)\big)_{i,j} =
  \begin{cases}
    \frac{1}{(i-j)!} B_{i-j}(1! a_1, 2! a_2, \ldots, (i-j)! a_{i-j}), & i \geq j \\
    0, & \text{otherwise,}
  \end{cases}
\]
where $B_k$ is the (exponential) \word{Bell polynomial} of degree $k$, defined by  $B_0 = 1$ and
\[
  B_i(x_1, \ldots, x_i) = i! \sum_{j_1+2j_2 + \cdots + i j_i = i} \frac{1}{j_1! j_2! \cdots j_i!}
  \left(\frac{x_1}{1!}\right)^{j_1}
  \left(\frac{x_2}{2!}\right)^{j_2} \cdots
  \left(\frac{x_i}{i!}\right)^{j_i}
  \quad \text{for $i \geq 1$.}
\]
This description implies the following facts.
\begin{prop}
  The action of $\Ga^n$ on $\P^n$ partitions $\P^n$ into orbits
  \[ O_k \coloneqq \{ b_0 = b_1 = \cdots = b_{n-k-1} = 0, \, b_{n-k} \neq 0 \}, \quad 0 \leq k \leq n. \]
  The stabilizer of any point of $O_k$ is
  \[ \St_k \coloneqq \{ a_1 = a_2 = \cdots = a_{k} = 0 \} \subset \Ga^n, \]
  and the closure of $O_k$ in $\P^n$ is $\overline{O_k} = \cup_{i \leq k} \,O_i$.
\end{prop}
\begin{zb}
  The action of $\G_a^2$ on $\P^2$ is 
  \[ (a_1, a_2) \cdot [b_0 : b_1 : b_2] =
    \begin{pmatrix}
      1 & 0 & 0 \\
      a_1 & 1 & 0 \\
      \frac12 a_1^2 + a_2 & a_1 & 1
    \end{pmatrix}
    \cdot [b_0 : b_1 : b_2]
    = \left[b_0 : a_1 b_0 + b_1 : \tfrac12 a_1^2 b_0 + a_2 b_0 + a_1 b_1 + b_2 \right]. \]
\end{zb}

\subsection{}\label{rmk_Gm} Consider now the semidirect product  $\Ga^n \rtimes \Gm$ \revision{(with $n \geq 1$)}, in which $\Gm$ acts on $\Ga^n$ by $t \cdot \mathbf a = (t a_1, t^2 a_2, \ldots, t^n a_n)$.
    If we define
\[
 \lambda_n: \Gm \to \Aut(\P^n), \quad  t \mapsto \begin{pmatrix}
                                                1 &   & &&\\
                                                  & t &  &&\\
                                                  & & t^2 && \\
                                                  &&&\ddots & \\
                                                &&&& t^n
                                               \end{pmatrix} \in \Aut(\P^n),
\]
then $\rho_n$ extends to an injective homomorphism
\[ \Ga^n \rtimes \Gm \hookrightarrow \Aut(\P^n), \quad (\mathbf a,t) \mapsto \rho_n(\mathbf a) \lambda_n(t), \]
in which the image of $\lambda_n$ normalizes the image of $\rho_n$.
The orbits of $\Ga^n \rtimes \Gm$ on $\P^n$ are the same as those of $\Ga^n$, but each $\Ga^n$-orbit $O_i$ contains a unique $\Gm$-fixed point. In coordinates, the fixed point is $[\mathbf e_{n+1-i}]$. Consequently, we may canonically identify $O_i$ with $\Ga^n / \St_i$ via the map $\mathbf a \mapsto \rho_n(\mathbf a) \cdot [\mathbf e_{n+1-i}]$.

\subsection{} Throughout the remainder of \cref{sec:polymatroidschubert}, fix notation:
\begin{itemize}
\item $V \subset \K^{\mathbf n} = \K^{n_1} \times \cdots \times \K^{n_N}$ is a $d$-dimensional linear subspace,
\item $a_{ij}$ and $b_{ik}$, with $1 \leq i \leq N$, $1 \leq j \leq n_i$, and $0 \leq k \leq n_i$, are coordinates on $\K^{\mathbf n}$ and $\P^{\mathbf n}$, respectively,
\item $P$ is the polymatroid on $E = \{1, \ldots, N\}$ associated to $V$,
\item $\widetilde P$ is the lift of $P$ with respect to $\mathbf n$, on ground set $\widetilde E = E_1 \sqcup \cdots \sqcup E_N$.
\end{itemize}
Phrases like ``combinatorial flat of $V$'' should be taken to mean ``combinatorial flat of $(P, \mathbf n)$.''
In Sections \ref{geomproofs} and \ref{sec:topresults}, we will assume that $V$ is ``polymatroid general'', in the sense defined in \cref{pgprops}.

\subsection{}\label{polymatschubertdfn}
We construct Schubert varieties of subspace arrangements, bootstrapping from the Hassett-Tschinkel action of \cref{sec_HT}.
\revision{Identifying $\G_a^{n_i}$ and $\K^{n_i}$,} the Hassett-Tschinkel maps $\rho_{n_i}: \K^{n_i} \to \Aut(\P^{n_i})$ assemble to a map
$\rho: \K^{\mathbf n} \to \Aut(\P^{\mathbf n})$, allowing $\K^{\mathbf n}$ to act on $\P^{\mathbf n}$.
For each multiset $\mathbf s \leq \mathbf n$, there is an $|\mathbf s|$-dimensional orbit
$O_{\mathbf s} := O_{s_1} \times \cdots \times O_{s_N} \subset \P^{\mathbf n}$, whose closure is a product of projectivized coordinate subspaces $\P^{\mathbf s} := \P^{s_1} \times \cdots \times \P^{s_N} \subset \P^{\mathbf n}$.
All points in $O_{\mathbf s}$ have the same stabilizer, denoted by $\St_{\mathbf s}$.
The explicit description in \cref{sec_HT} shows that $\St_{\mathbf s}$ is the coordinate subspace of $\K^{\mathbf n}$ cut out by
\[
  a_{i,1} = a_{i,2} = \cdots = a_{i,s_i} = 0, \quad 1 \leq i \leq N.
\]
By \cref{rmk_Gm}, $\rho$ extends to an injection
\[ \K^{\mathbf n}  \rtimes \Gm \hookrightarrow \Aut(\P^{\mathbf n}), \]
and each orbit $O_{\mathbf s}$ contains a unique $\Gm$-fixed point $([\mathbf e_{n_1 +1 - s_1}], \ldots, [\mathbf e_{n_N +1 - s_N}])$.
This yields a canonical isomorphism
\[ \iota: \K^{\mathbf n} \to O_{\mathbf n},  \quad \mathbf a \mapsto \rho(\mathbf a) \cdot ([\mathbf e_{1}], \ldots, [\mathbf e_{1}]). \]

\begin{dfn}
  The \word{Schubert variety} of $V \subset \K^{\mathbf n}$ is $Y_V := \overline{\iota(V)}$, the closure of $\iota(V)$ in $\P^{\mathbf n}$.
\end{dfn}

\subsection{}\label{chowclass}
We compute the support of the Chow class of $Y_V$.
Given a nonsingular variety $X$, let $A_d(X)$ be the Chow group of $X$ spanned by $d$-dimensional cycles.
If $Y$ and $Y'$ are subvarieties of $X$, then their intersection product is written $[Y] \cdot [Y']$.
\begin{lem}\label{chowbases}
  The Chow class of $Y_V$ in $A_d(\P^{\mathbf n})$ is of the form
  \[
    [Y_V] = \sum_{\mathbf b \text{ basis of $P$}} c_{\mathbf b}\, [\P^{\mathbf b}].
  \]
\end{lem}
\begin{proof}
  Write $[Y_V] =\sum_{|\mathbf t| = d} c_{\mathbf t} [\P^{\mathbf t}]$.
  Let $\mathbf s$ be a multiset with $|\mathbf s| = d$ that is not a basis of $P$. We will show that $[\P^{\mathbf s}]$ has coefficient zero in $[Y_V]$.

  Let $T = \{i : \overline s_i = n_i\}$. This set is nonempty: if we suppose falsely that $T = \emptyset$, then by \cref{combrankformula}, $\mathbf s = \overline{\mathbf s}$. \revision{This means that $\mathbf s$ is a combinatorial flat satisfying $\mathbf s^\geo = \mathbf 0$, so by \cref{combrankformula} once more,
  \[ \rk_P(\mathbf s) = \rk_P(\mathbf s^\geo) + |\mathbf s - \mathbf s^\geo| = 0 + |\mathbf s| = d, \]
  contrary to the fact that $\mathbf s$ is not a basis.}

\revision{Let $\mathbf s' = \mathbf s - \sum_{i \in E \setminus T} s_i \mathbf e_i$.
The pushforward $\pi_{T*}: A_0(\P^{\mathbf n}) \to A_0(\P^{\pi_T(\mathbf n)})$ is an isomorphism, so $c_{\mathbf s} = 0$ if and only if $\pi_{T*}(c_{\mathbf s}[\pt])$ is zero in $A_*(\P^{\pi_T(\mathbf n)})$.
By the projection formula,
\[
  \pi_{T*}(c_{\mathbf s}[\pt]) = \pi_{T*}([\P^{\mathbf n - \mathbf s}]  \cdot [Y_V]) = \pi_{T*}([\P^{\mathbf n - \mathbf s'}] \cdot [\P^{\mathbf n - \mathbf s + \mathbf s'}] \cdot [Y_V]) = [\P^{\pi_T(\mathbf n - \mathbf s')}] \cdot \pi_{T*}([\P^{\mathbf n - \mathbf s + \mathbf s'}] \cdot [Y_V]).
\]
The class $[\P^{\mathbf n-\mathbf s + \mathbf s'}] \cdot [Y_V]$ is represented by a subvariety $Y'$ of $Y_V$, obtained by intersecting $Y_V$ with appropriately general hyperplanes.
To show that $[\P^{\pi_T(\mathbf n - \mathbf s')}] \cdot \pi_{T*}([\P^{\mathbf n - \mathbf s + \mathbf s'}] \cdot [Y_V]) = [\P^{\pi_T(\mathbf n - \mathbf s')}] \cdot \pi_{T*}([Y'])$ is zero, it suffices to show that $\dim \pi_T(Y') < |\pi_T(\mathbf s')|$.
   In fact, this is the case: by \cref{combrankformula},
  \begin{align*}
    \dim \pi_T(Y') \leq \dim \pi_T(Y_V) &= \rk_P(T) = \rk_{P, \mathbf n}(\overline{\mathbf s}^\geo) \\
                                                    &= \rk_{P, \mathbf n}(\overline{\mathbf s}) - \sum_{i \in E \setminus T} s_i < d - \sum_{i \in E \setminus T} s_i= \sum_{i \in T} s_i = |\pi_T(\mathbf s')|.
  \end{align*}
  We conclude that $c_{\mathbf s} = 0$.}
\end{proof}
\begin{lem}\label{chowclass:positivity}
  \revision{For each basis $\mathbf b$ of $P$, $c_{\mathbf b} > 0$.}
\end{lem}
\begin{proof}
  Let $I: \K^{\mathbf n} \to O_{\mathbf n} \subset \P^{\mathbf{n}}$ be the usual map from $\prod_{i=1}^N \K^{n_i}$ to $\prod_{i=1}^N\P^{n_i}$, given by $\frac{b_{ij}}{b_{i0}}(I(\mathbf a)) = a_{ij}$.
  Write $\overline V$ for the closure of $I(V)$ in $\P^{\mathbf n}$. The Chow class of $\overline V$ is (see e.g. \cite{CFCYL20}) given by
  \[
    [\overline V] = \sum_{\text{$\mathbf b$ basis of $P$}} [\P^{\mathbf b}].
  \]
  Therefore if we can find a flat family $X_t$ of subvarieties of $\P^{\mathbf{n}}$ such that $X_1$ is isomorphic to $Y_V$ and $X_0 \supset \overline V$, then $c_{\mathbf b} \geq 1$ for all bases $\mathbf b$ of $P$.   Recall from \cref{polymatschubertdfn} that $Y_V$ is the closure of $\iota(V)$, where $\iota: \K^{\mathbf n} \to O_{\mathbf n}$. The idea is that $I(V)$ is the tangent space to $\iota(V)$ at the origin of $O_{\mathbf{n}}$, so we can construct a deformation of $Y_V$ by taking the closure of the deformation of $\iota(V)$ to its tangent space.

  We now construct such a family. Allow $\G_m$ to act on $\K^{\mathbf n}$ and $\P^{\mathbf n}$ in the standard way, given in the $i$th factor by
  \begin{align*}
    t \cdot (a_{i,1},\ldots, a_{i,n_i}) &= (t a_{i,1}, \ldots, t a_{i,n_i}) \quad \text{and} \\
    t \cdot [b_{i,0}: b_{i,1}: \cdots : b_{i,n_i}] &= [b_{i,0}: t b_{i,1} : \cdots : t b_{i,n_i}],
  \end{align*}
  and define
  \[
    \iota_\circ': \K^{\mathbf n} \times \Gm \to O_{\mathbf n} \times \A^1, \quad (\mathbf a, t) \mapsto t^{-1} \cdot \iota(t \cdot \mathbf a).
  \]
  From \cref{sec_HT}, we see that
\revision{  \[
    \frac{b_{i\ell}}{b_{i0}}(\iota'_\circ(\mathbf a, t)) = \sum_{j_1+2j_2 + \cdots + \ell j_\ell = \ell} t^{j_1+j_2+\cdots+j_\ell - 1} \frac{a_{i,1}^{j_1} a_{i,2}^{j_2} \cdots a_{i,\ell}^{j_\ell}}{j_1! j_2! \cdots j_\ell!}, \quad 1 \leq \ell \leq n_i,
  \]}
  so $\iota_\circ'$ extends to a regular map
  \[
    \iota': \K^{\mathbf n} \times \A^1 \to O_{\mathbf n} \times \A^1
  \]
  satisfying
  \[
    \frac{b_{i\ell}}{b_{i0}}(\iota'(\mathbf a, 0)) = a_{i\ell}.
  \]
  Let $X$ be the closure of $\iota'(V \times \A^1)$ in $\P^{\mathbf n} \times \A^1$. \revision{Equivalently, $X$ is the closure of $\iota'_\circ(V \times \G_m)$, so $X$ is a flat family over $\A^1$ (e.g. by \cite[Proposition 4.3.9]{liu}).}

  Observe that $X_0 \supset \overline V$.
  On the other hand, there is a diagram
  \[\begin{tikzcd}
      V \times \Gm \ar[r,hook] \ar[d,"\cong"] & \K^{\mathbf n} \times \Gm \ar[r, "\iota' \times \Id"]\ar[d,"\cong"] & \P^{\mathbf n} \times \Gm \ar[d, "\cong"] \\
      V \times \Gm \ar[r,hook] & \K^{\mathbf n} \times \Gm \ar[r, "\iota \times \Id"] & \P^{\mathbf n} \times \Gm \\
    \end{tikzcd} \quad\text{given fiberwise by}\quad
  \begin{tikzcd}
    V \ar[r, hook]\ar[d, "t \cdot"] & \K^{\mathbf n}_t \ar[r, "\iota'"] \ar[d,"t\cdot"] & \P^{\mathbf n}_t \ar[d, "t\cdot"]\\
    V \ar[r, hook] & \K^{\mathbf n}_t \ar[r, "\iota"] & \P^{\mathbf n}_t.
  \end{tikzcd}\]
Taking closures, we obtain an isomorphism of families
\[\begin{tikzcd}
    X \ar[r,hook] \ar[d,"\cong"] & \P^{\mathbf n} \times \Gm \ar[d,"\cong"] \\
    Y_V \times \Gm \ar[r,hook] & \P^{\mathbf n} \times \Gm,
    \end{tikzcd}
  \]
  that identifies $X_1$ with $Y_V$.
\end{proof}
\revision{
\begin{rmk}
  Another proof of \cref{chowclass:positivity} goes as follows: a projection of $O_{\mathbf n} \cap Y_V$ onto a given $d$-element subset of the coordinates $\{b_{ij} / b_{i0}\}_{i,j}$ is dominant if and only if its induced map on tangent spaces is surjective at a general point of $O_{\mathbf n} \cap Y_V$. Consequently, to prove that $c_{\mathbf b} > 0$, it suffices to produce a point $p \in O_{\mathbf n} \cap Y_V$ and a $d$-element set of coordinates representing $\mathbf b$ such that the derivative of the corresponding projection is surjective at $p$.

  The derivative of $\iota$ at $0 \in V$ identifies $V$ with the tangent space of $Y_V$ at the point $p = [1:0:0:\cdots:0] \times \cdots \times [1:0:0:\cdots:0]$.
  Therefore, if $\mathbf b$ is a basis of the polymatroid $P$ associated to $V$, then a $d$-element set of the type desired exists, proving that $c_{\mathbf b} > 0$.
\end{rmk}}

\subsection{}\label{pgprops}
The Schubert variety $Y_V$ is $V$-equivariant, hence is a union of $V$-orbits.
To guarantee that the number of orbits is finite, we require
$V$ to be \word{polymatroid general} (henceforth, \word{p.g.}), meaning that for each multiset $\mathbf s \leq \mathbf n$,
\[
  \rk_{P, \mathbf n}(\mathbf s) = \codim_V V \cap \St_{\mathbf s}.
\]
If the matroid of $V$ in $(\K^1)^{|\mathbf n|} \cong \K^{\mathbf n}$ is $\widetilde P$, then $V$ is polymatroid general, so any realizable polymatroid has a p.g. realization by \cref{realizablelift}.
However, not all p.g. realizations of $P$ realize $\widetilde P$.

The next two results allow us to make inductive arguments; we will use them often without mention.
\begin{prop}
  If $V$ is p.g. and $i \in E$, then $\pi_{E \setminus i}(V)$ is p.g.
\end{prop}
\begin{proof}
  Let $\mathbf s' \leq \pi_{E \setminus i}(\mathbf n)$ be a multiset and let
  $\mathbf s = \sum_{j \in E \setminus i} s_j' \mathbf e_j$.
   There is a commuting diagram with exact rows
  \[\begin{tikzcd}
      0 \ar[r] & V \cap \ker(\pi_{E \setminus i}) \ar[r]\ar[d,equal] & V \cap \St_{\mathbf s} \ar[r] \ar[d] & V' \cap \St_{\mathbf s'} \ar[r]\ar[d] &  0 \\
      0 \ar[r] & V \cap \ker(\pi_{E \setminus i}) \ar[r] & V \ar[r, "\pi_{E \setminus i}"] & V' \ar[r] & 0,
  \end{tikzcd}\]
so $\codim_{V'} V' \cap \St_{\mathbf s'} = \codim_V V \cap \St_{\mathbf s} = \rk_P(\mathbf s)$.
The multisets $\mathbf s$ and $\mathbf s'$ are represented by the same set $S \subset \widetilde E \setminus E_i$, so
\[ \codim_{V'} V' \cap \St_{\mathbf s'} = \rk_{P,\mathbf n}(\mathbf s) =\rk_{\widetilde P}(S) = \rk_{\widetilde P \setminus E_i}(S) = \rk_{\widetilde{P \setminus i}}(S) = \rk_{P\setminus i, \pi_{E \setminus i}(\mathbf n)}(\mathbf s'), \]
which is the desired conclusion.
\end{proof}
\begin{prop}\label{pgtruncation}
  Suppose $V \subset \K^{\mathbf n}$ is a p.g. subspace and $A \subset E$.
  If $H$ is the preimage in $\K^{\mathbf n}$ of a general\footnote{By which we mean a hyperplane such that $H \cap V$ realizes $T_A P$. See also \cref{opsgeometry}.}
  hyperplane in $\K^{\pi_A(\mathbf n)}$, then $V \cap H$ is p.g.
\end{prop}
\begin{proof}
  Let $\mathbf s \leq \mathbf n$ be a multiset represented by $S \subset \widetilde E$.
  From \cref{sec:truncation}, we learn
  \[
    \rk_{T_A P, \mathbf n}(\mathbf s) = \rk_{\widetilde{T_A P}}(S)
    =  \begin{cases} \rk_{\widetilde P}(S) - 1, & \text{if $\cup_{i \in A} E_i \subset \overline S$} \\
      \rk_{\widetilde P}(S), & \text{otherwise}
                               \end{cases}
                               = \begin{cases}
                                 \rk_{P, \mathbf n}(\mathbf s) - 1, & \text{if $A \subset \{i : \overline s_i = n_i\}$} \\
                                 \rk_{P,\mathbf n}(\mathbf s), & \text{otherwise.}
                                 \end{cases}
  \]
  On the other hand,
  \[
    \codim_{V \cap H} V \cap H \cap \St_{\mathbf s}
    = \begin{cases}
      \codim_V V \cap \St_{\mathbf s} - 1, & \text{if $V \cap \St_{\mathbf s} \subset V \cap H$} \\
      \codim_V V \cap \St_{\mathbf s}, & \text{otherwise}
    \end{cases}
    = \begin{cases}
      \rk_{P, \mathbf n}(\mathbf s) - 1, & \text{if $V \cap \St_{\mathbf s} \subset V \cap H$} \\
      \rk_{P, \mathbf n}(\mathbf s), & \text{otherwise.}
     \end{cases}                                  
   \]
   To finish, observe that $A \subset \{i : \overline s_i = n_i\}$ if and only if  $V \cap \St_{\mathbf s} \subset V \cap H$ for a Zariski-general $H$.
\end{proof}

\subsection{}\label{geomproofs}
We establish \cref{main:stratification} via a series of lemmas characterizing the strata of $Y_V$.
Throughout this section, $V$ is assumed to be polymatroid general.
\begin{lem}\label{avoidance}
  If $\mathbf s$ is not a combinatorial flat of $V$, then $Y_V \cap O_{\mathbf s}$ is empty.
\end{lem}
\begin{proof}
  We induct on $N$.
  For any point $x \in O_{\mathbf s}$,
  \begin{equation}
    \label{dimformula}\tag{$*$}
    \dim(V \cdot x) = \codim_V(V \cap \St_{\mathbf s}) = \rk_P(\mathbf s).
  \end{equation}
  Since $Y_V$ is a union of $V$-orbits and $V$ is p.g., the desired statement follows immediately from dimensional considerations when $N=1$.

  Otherwise, suppose $N > 1$.
  If $\rk(\mathbf s) = d$, then the dimension formula implies $Y_V \cap O_{\mathbf s}$ is nonempty if and only if $\mathbf s = \mathbf n$.
  Otherwise, if $\rk \mathbf s < d$, then there is $i \in E$ such that  $\overline{s}_i < n_i$.
  The projection $Y_V \to Y_{\pi_{E \setminus i}(V)}$ sends $Y_V \cap O_{\mathbf s}$ into $Y_{\pi_{E \setminus i}(V)} \cap O_{\pi_{E \setminus i}(\mathbf s)}$.
  By \cref{localproduct}, $\pi_{E \setminus i}(\mathbf s)$ is not a combinatorial flat of $\pi_{E \setminus i}(V)$, so  $Y_{\pi_{E \setminus i}(V)} \cap O_{\pi_{E \setminus i}(\mathbf s)}$ is empty by induction on $N$. This means $Y_V \cap O_{\mathbf s}$ is empty too.
\end{proof}

The following result gives a sharp description of the geometry of a nonempty stratum.
\begin{lem}\label{strata}
  If $\mathbf s$ is a combinatorial flat of $V$ such that $Y_V \cap O_{\mathbf s}$ is nonempty and $V' = \pi_{\{i : s_i = n_i\}}(V)$, then
  $ Y_V \cap O_{\mathbf s} = \iota(V') \times \prod_{s_i < n_i} O_{s_i}$.
  Consequently, $Y_V \cap O_{\mathbf s}$ is a single $V$-orbit of dimension $\rk(\mathbf s)$, and
  \[ \textstyle \overline{Y_V \cap O_{\mathbf s}}  = Y_{V'} \times \prod_{s_i < n_i} \P^{s_i} = Y_{V' \times \prod_{s_i < n_i} \K^{s_i}} \subset \P^{\mathbf s}. \]
\end{lem}
\begin{proof}
  The statement plainly holds when $\mathbf s = \mathbf n$.
  Otherwise, assume $\mathbf s$ is a proper combinatorial flat.
  The $V$-orbit of any point in $O_{\mathbf s}$ has dimension $\rk(\mathbf s)$ and $Y_V$ is $V$-equivariant, so $Y_V \cap O_{\mathbf s}$ has dimension at least $\rk(\mathbf s)$.
  
  We complete the proof for proper combinatorial flats by inducting on $N$.
  First consider $N = 1$, so that  $Y_V \subset \P^{n_1}$.
  By \cref{avoidance}, $Y_V \cap \P^{n_1 - 1} \subset \P^{\dim V - 1}$.
  Since $Y_V \cap \P^{n_1 - 1}$ is a divisor on $Y_V$, we must have $Y_V \cap \P^{n_1 - 1} = \P^{\dim V - 1}$.
  Since $V$ is polymatroid general, $O_i \subset \P^{\dim V - 1}$ is a single $V$-orbit, and $\overline{Y_V \cap O_i} = \P^{i}$ as desired.

  Now, suppose $N > 1$.
  Since $\mathbf s$ is a proper combinatorial flat, there is $j \in E$ such that $s_j < n_j$.
  Set $V' = \pi_{E \setminus j}(V)$ and $\mathbf s' = \pi_{E \setminus j}(\mathbf s)$.
  The projection $\pi_{E \setminus j}: \P^{\mathbf n} \to \P^{\pi_{E \setminus i}(\mathbf n)}$ induces a surjection $Y_V \to Y_{V'}$. This map sends $Y_V \cap O_{\mathbf s}$ into $Y_{V'} \cap O_{\mathbf s'}$, so by the induction hypothesis
  \[ \textstyle Y_V \cap O_{\mathbf s} \subset O_{s_j} \times (Y_{V'} \cap O_{\mathbf s'})
    = O_{s_j} \times \iota(\pi_{\cup_{s_i' = n_i} E_i}( V)) \times \prod_{s_i' < n_i} O_{s_i}
    = \iota(\pi_{\cup_{s_i = n_i} E_i}( V)) \times \prod_{s_i < n_i} O_{s_i}.
  \]
  The dimension of the left-hand set is at least $\rk(\mathbf s)$ because $V$ is polymatroid general; the right-hand set is connected of dimension  $\rk(\mathbf s)$ by \cref{combrankformula}.
  Hence, the two sides are equal.

  To see that $Y_V \cap O_{\mathbf s}$ is a single $V$-orbit, recall that $V$ is polymatroid general, meaning that the $V$-orbit of any point in $O_{\mathbf s}$ has dimension $\rk(\mathbf s)$, and we have just proved that $Y_V \cap O_{\mathbf s}$ is connected of dimension $\rk(\mathbf s)$.
  The remainder of the ``consequently'' is clear.
\end{proof}
\begin{lem}\label{inclusion}
If $\mathbf s$ is a combinatorial flat, then $Y_V \cap O_{\mathbf s}$ is nonempty.  
\end{lem}
\begin{proof}
  We induct on $\dim V$.
  If $\dim V = 1$, let $\mathbf s$ be its minimal combinatorial flat.
  Since $V$ is just a line, one sees that $O_{\mathbf s} \subset Y_V$ using the explicit description of the Bell polynomials in \cref{sec_HT}.

  Otherwise, suppose $\dim V > 1$.
  For a general hyperplane $H$ in $\K^{\mathbf n}$, $Y_{V \cap H} \cap O_{\mathbf s}$ is nonempty for every combinatorial flat $\mathbf s$ of corank greater than 1 by the induction hypothesis and \cref{truncation}.
  By the same argument, if $F = \{i : s_i = n_i\}$ is nonempty, then  $Y_{V \cap H} \cap O_{\mathbf s}$ is nonempty for $H \subset \K^{\mathbf n}$ the preimage of a general hyperplane in $\K^{\pi_F(\mathbf n)}$.

  We have now reduced to the case where $\mathbf s$ is a combinatorial flat of rank $d-1$ with $\mathbf s^\geo = \mathbf 0$.
  By \cref{indeprank}, this means that $\mathbf s$ is an independent multiset of cardinality $d-1$, so there is $i \in E$ such that $\mathbf s + \mathbf e_i$ is a basis of $P$.
  Let $H = \P^{\mathbf n - \mathbf e_i}$.
  Every irreducible component of $Y_V \cap H$ is $(d-1)$-dimensional, so by \cite[Theorem 1.26]{EH16},
  \[
    [\P^{\mathbf n - \mathbf e_i}] \cdot [Y_V] = \sum_C m_C [C],
  \]
  where $C$ runs over all irreducible components of $Y_V \cap H$.
  By \cref{avoidance}, $Y_V \cap H \subset \cup_{\mathbf s'} Y_V \cap \P^{\mathbf s'}$, where the union runs over all corank 1 combinatorial flats $\mathbf s' \leq \mathbf n - \mathbf e_i$.
  By \cref{strata}, each irreducible component $C$ of $Y_V \cap H$ is among those of $\cup_{\mathbf s'} Y_V \cap \P^{\mathbf s'}$, so
  \[
        [\P^{\mathbf n - \mathbf e_i}] \cdot [Y_V] = \sum_{\mathbf s'} m_{\mathbf s'} [Y_V \cap \P^{\mathbf s'}],
  \]
  with $m_{\mathbf s'} = 0$ if $\dim Y_V \cap \P^{\mathbf s'} < d-1$.
  By \cref{strata} and \cref{chowclass:positivity}, we may further expand as
  \[
        [\P^{\mathbf n - \mathbf e_i}] \cdot [Y_V] = \sum_{\mathbf s'} m_{\mathbf s'} \sum_{\mathbf b' \text{ basis of $\mathbf s'$}} c_{\mathbf s', \mathbf b'} [\P^{\mathbf b'}].
  \]
  
   On the other hand, by \cref{chowclass:positivity},
   \[
      [\P^{\mathbf n - \mathbf e_i}] \cdot [Y_V] = \sum_{\substack{\mathbf b \text{ basis of $P$}\\b_i > 0}} c_{\mathbf b} [\P^{\mathbf b - \mathbf e_i}],
   \]
   with all coefficients positive.
   Since $\mathbf s$ is both an independent multiset and a combinatorial flat of $P$,  $[\P^{\mathbf s}]$ appears in both expansions, and $0 < c_{\mathbf s} = m_{\mathbf s} c_{\mathbf s, \mathbf s}$.
   This implies $m_{\mathbf s} \neq 0$, so $\dim Y_V \cap \P^{\mathbf s} = d-1$, hence $Y_V \cap O_{\mathbf s} \neq \emptyset$.
\end{proof}
\begin{proof}[Proof of \cref{main:stratification}]
  Characterization of when $Y_V$ intersects $O_{\mathbf s}$ is \cref{avoidance} and \cref{inclusion}.
  \cref{main:stratification}\cref{main:stratification:orbit} follows from \cref{strata}, and \cref{main:stratification}\cref{main:stratification:class} is obtained by combining \cref{strata} and \cref{chowclass:positivity}.

  It remains to prove \cref{main:stratification}\cref{main:stratification:poset}, which states that if $\mathbf s, \mathbf s' \in \mathcal L_{P, \mathbf n}$, then  $\overline{Y_V \cap O_{\mathbf s}} \supset Y_V \cap O_{\mathbf s'}$ if and only if $\mathbf s \geq \mathbf s'$.
  
  Fix $\mathbf s \in \mathcal L_{P, \mathbf n}$.
  By \cref{localproduct}, $\mathbf s' < \mathbf s$ is a combinatorial flat of $V$ if and only if $\mathbf s_i' \leq \mathbf s_i$ for all $i$ and $(\mathbf s_j')_{j :s_j = n_j}$ is a combinatorial flat of $V' := \pi_{\{j : s_j = n_j\}}(V)$. 
  Such multisets $\mathbf s'$ are precisely those that index the nonempty strata of $\overline{Y_V \cap O_{\mathbf s}}$, since $\overline{Y_V \cap O_{\mathbf s}} = Y_{V'} \times \prod_{s_i < n_i} \P^{s_i}$ by \cref{strata}.
\cref{strata} also implies  $O_{\mathbf s'} \cap (Y_{V'} \times \prod_{s_i < n_i} \P^{s_i}) = O_{\mathbf s'} \cap Y_V$ whenever both sides are nonempty, which completes the proof.
\end{proof}

\begin{rmk}[Rescaling coordinates]
  Theorems \ref{main:stratification} describes properties of Schubert varieties of subspace arrangements, extending well-known results for Schubert varieties of hyperplane arrangements from \cite{AB16}, which correspond to the case $\mathbf n = (1, \ldots, 1)$.
  One further such property follows: when $\mathbf n = (1, \ldots, 1)$,  it is well-known that changing $V$ by rescaling each coordinate of $\K^{\mathbf n} = \K^N$ does not change the isomorphism class of $Y_V$.
  For general $\mathbf n$, the isomorphism class of $Y_V$ is unchanged if we perform a weighted rescaling of each factor $\K^{n_i}$, as described in \cref{rmk_Gm}.
\end{rmk}
\begin{rmk}[Flags from the group action]\label{flagperspective}
  Assume that $V$ realizes a simple polymatroid.
  The subspaces $V_{ij} = V \cap \St_{j \mathbf e_i}$ for $1 \leq i \leq N$ and $1 \leq j \leq n_i$ comprise a collection of $N$ flags in $V$. The combinatorial flats record intersection pattern of these flags, as $\rk(\mathbf s) = \codim_V V_{1,s_1} \cap V_{2,s_2} \cap \cdots \cap V_{N,s_N}$.
The ``polymatroid general'' condition on $V$ guarantees that these flags are in general position with respect to one another (see also \cref{geometryofcf}).
\end{rmk}

\begin{zb}[Simplification changes the Schubert variety]
    Let $\mathbf{n} = (2,2)$, and consider the subspace $V \subset \K^{\mathbf{n}} = \K^2 \times \K^2$ given by the rowspan of
    \[
    \begin{pmatrix}
        1 & 0 & 1 & 1\\
        0 & 1 & 1 & -1
    \end{pmatrix}.
    \]
    The data $V \subset \K^{\mathbf{n}}$ realizes the nonsimple polymatroid of \cref{zb:simplification}, whose simplification is the boolean matroid on two elements. However, we will see that $Y_V$ is not isomorphic to the Schubert variety $(\P^1)^2$ of the simplification. Consider the linear paramitrization $(u,v) \mapsto (u,v,u+v,u-v)$ of $V$, and the corresponding paramatrization of $\iota(V) \subset \P^2 \times \P^2$:
    \[
    (u,v) \mapsto [1:u:v + \frac12 u^2] \times [1:u+v: u-v + \frac12(u+v)^2].
    \]
    Recall that we label the coordinates on the right hand side as $[b_{10}:b_{11}:b_{12}] \times [b_{20}:b_{21}:b_{22}]$. The relations between the coordinates $b_{11},b_{12},b_{21},b_{22}$ on the right hand side are generated by the inhomogenious equations
    \begin{align*}
        &b_{21}^2 + 4b_{11} - 2b_{21} - 2b_{22},\\
        &b_{11}^2 - 2b_{11} - 2b_{12} + 2b_{21}.
    \end{align*}
    To take the closure in $\P^{\mathbf{n}}$, we homogenize the above equations in the variables $b_{10},b_{20}$ so that they are bihomogeneous, and then saturate them with respect to the ideal generated by $b_{10}b_{20}$ to obtain the homogenious equations
    \begin{align*}
        &160\,b_{11}b_{20}^{2}-32\,b_{12}b_{20}^{2}-96\,b_{10}b_{20}b_{21}+10\,b_{20}b_{21}-b_{21}^{2}+2\,b_{20}b_{22},\\
        &10\,b_{10}b_{11}b_{20}+b_{11}^{2}b_{20}-2\,b_{10}b_{12}b_{20}-6\,b_{10}^{2}b_{21},\\
        &240\,b_{11}^{3}b_{20}+48\,b_{11}^{2}b_{12}b_{20}-96\,b_{10}b_{12}^{2}b_{20}-1\,440\,b_{10}^{2}b_{11}b_{21}-288\,b_{10}^{2}b_{12}b_{21}+1\,500\,b_{10}^{2}b_{21}\\
        &\quad\quad\quad-250\,b_{10}b_{11}b_{21}-25\,b_{11}^{2}b_{21}+50\,b_{10}b_{12}b_{21}+300\,b_{10}^{2}b_{22}.\\
    \end{align*}
    One can check (e.g. by using Macaulay2 \cite{M2}) that the subvariety $Y_V \subset\P^{\mathbf{n}}$ defined by the displayed equations is singular, and therefore not isomorphic to $(\P^1)^2$. 
\end{zb}
\subsection{}\label{sec:topresults}
In this section, we establish \cref{cohomology2}, which gives a formula for the cohomology ring of $Y_V$.
Throughout, we assume that $V$ is polymatroid general.
For an independent multiset $\mathbf b$, let $c_{\mathbf b} > 0$ be the coefficient of $[\P^{\mathbf b}]$ in   $[Y_V \cap \P^{\overline{\mathbf b}}]$, as in the statement of \cref{main:stratification}.
\begin{prop}\label{cohomology1}
  If $V \subset \C^{\mathbf n}$ is polymatroid general, then its singular cohomology ring $H^*(Y_V, \Q)$ is isomorphic to $\Q[y_1, \ldots, y_N] / I$, where $I$ is the ideal generated by 
  \begin{align*}
    \textstyle    c_{\mathbf b'} y_1^{b_1} \cdots y_N^{b_N} - c_{\mathbf b} y_1^{b_1'} \cdots y_N^{b_N'}, &\quad \text{$\mathbf b$ and $\mathbf b'$ are independent multisets of $P$ with $\overline{\mathbf b} = \overline{\mathbf b'}$, and } \\
    y_1^{d_1}\cdots y_N^{d_N}, &\quad \text{$\mathbf d$ is a dependent multiset.}
  \end{align*}
\end{prop}
\begin{proof}
  Let $\kappa: Y_V \to \P^{\mathbf n}$ denote the inclusion.
  \revision{By \cref{main:stratification}\cref{main:stratification:class}, the classes $\{[Y_V \cap \P^{\mathbf s}] : \mathbf s \in \mathcal L_{P,\mathbf n}\}$ are linearly independent in $A_* \P^{\mathbf n}$, so the pushforward $A_* Y_V \to A_* \P^{\mathbf n}$ is injective.}

  The sets $O_{\mathbf s}$ and $Y_V \cap O_{\mathbf s}$ comprise \revision{algebraic} cell decompositions of $\P^{\mathbf n}$ and $Y_V$, respectively.
  Hence, the Chow groups of these varieties are naturally isomorphic to their Borel-Moore homology groups via the cycle class map \cite[Example 19.1.11]{F98}.
  Since $\P^{\mathbf n}$ and $Y_V$ are compact, their Borel-Moore homology is equal to their singular homology.
  Summing up, we've learned that the singular homology groups of $Y_V$ are all free and that the homology pushforward $\kappa_*: H_*(Y_V; \Z) \to H_*(\P^{\mathbf n}; \Z)$ is injective.
  Applying the Universal Coefficient Theorem, we learn that $\kappa^*: H^*(\P^{\mathbf n}; \Q) \to H^*(Y_V; \Q)$ is surjective.
  
  The cohomology ring of $\P^{\mathbf n}$ is isomorphic to $R := \Q[y_1, \ldots, y_N] / (y_1^{n_1+1}, \ldots, y_N^{n_N+1})$, where $y_i$ represents a hyperplane pulled back from $\P^{n_i}$. The monomials $y_i^{n_i+1}$ are among the claimed relations because $n_i \geq \rk_P(i)$ for all $1 \leq i \leq N$.

  Let us verify that the remaining generators of $I$ are in $\ker(\kappa^*)$. By the Universal Coefficient Theorem, the pullback of $\alpha \in H^k(\P^{\mathbf n}; \Q)$ is zero if and only if its cap product with any element of $H_k(Y_V; \Q)$ is zero.
  By the projection formula and the injectivity of $\kappa_*$, this is equivalent to $\alpha \smallfrown [\overline{Y_V \cap O_{\mathbf s}}] = 0$ for all rank $k$ combinatorial flats $\mathbf s$ of $P$.
  By \cref{main:stratification}\cref{main:stratification:class},
  \[ y_1^{b_1} \cdots y_N^{b_N} \smallfrown [\overline{Y_V \cap O_{\mathbf s}}] =
    \begin{cases}
        c_{\mathbf b}, & \text{if $\mathbf b$ is independent and $\overline{\mathbf b} =\mathbf s$} \\
      0, &\text{otherwise,}
    \end{cases} \]
  so all claimed relations are in $\ker(\kappa^*)$.
  Moreover, the dimension of the degree $k$ homogeneous component of $R/I$ is plainly equal to $\dim H^k(Y_V; \Q)$, which completes the proof.
\end{proof}

\begin{proof}[Proof of \cref{cohomology2}]
  Let $S = \Q[y_{\mathbf s} : \text{$\mathbf s \in \mathcal L_{P, \mathbf n}$}]$, graded by $\deg y_{\mathbf s} = \rk_P(\mathbf s)$.
  Using the notation of \cref{cohomology1}, define a map
  \[
    \phi: S \to H^*(Y_V, \Q), \quad y_{\mathbf s} \mapsto \frac{1}{c_{\mathbf b}} y_1^{b_1} \cdots y_N^{b_N}
  \]
  where $\mathbf b$ is any basis of $\mathbf s$. The image of $y_{\mathbf s}$ does not depend on $\mathbf b$ by \cref{cohomology1}.

  Let $J$ be the ideal of $S$ generated by
  \begin{align*}
    y_{\mathbf s} y_{\mathbf s'}, &\quad \rk(\mathbf s \smile \mathbf s') < \rk(\mathbf s) + \rk(\mathbf s') \\
    c_{\mathbf b} c_{\mathbf b'} y_{\mathbf s} y_{\mathbf s'} - c_{\mathbf b + \mathbf b'} y_{\mathbf s \smile \mathbf s'}, & \quad \text{$\mathbf b, \mathbf b'$ are bases of $\mathbf s, \mathbf s'$ so that $\mathbf b + \mathbf b'$ is a basis for $\mathbf s \smile \mathbf s'$.}
  \end{align*}
  We first show that $J \subset \ker \phi$. Suppose $\mathbf s$ and $\mathbf s'$ are two combinatorial flats.
  If $\rk(\mathbf s \smile \mathbf s') < \rk(\mathbf s) + \rk(\mathbf s')$, then the sum of any pair of bases of $\mathbf s$ and $\mathbf s'$ is a dependent multiset, so $\phi(y_{\mathbf s} y_{\mathbf s'}) = 0$.
  On the other hand, if $\rk(\mathbf s \smile \mathbf s') = \rk(\mathbf s) + \rk(\mathbf s')$, then there are bases $\mathbf b$ and $\mathbf b'$ such that $\mathbf b + \mathbf b'$ is a basis for $\mathbf s \smile \mathbf s'$, so $c_{\mathbf b} c_{\mathbf b'} \phi(y_{\mathbf s} y_{\mathbf s'}) = c_{\mathbf b + \mathbf b'} \phi(y_{\mathbf s \smile \mathbf s'})$.
  Hence, $J \subset\ker \phi$.

  Surjectivity of $\phi$ follows from the fact that $\phi(y_{\mathbf s}) \smallfrown [Y_V \cap \P^{\mathbf s'}]$ is 1 if $\mathbf s = \mathbf s'$ and 0 otherwise. Finally, it is evident from the relations that the $y_{\mathbf{s}}$ generate $S/J$ as a $\mathbb{Q}$-vector space, so the degree $k$ part of $S/J$ has dimension at most $\dim H^k(Y_V; \Q)$, and therefore $J =\ker \phi$.
\end{proof}

\begin{rmk}[Geometry \& top-heaviness of combinatorial flats]
  In general, our proof of \cref{combo}\cref{combo:topheavy} (\cref{comboflatsimplification}) relies on high-powered results of \cite{BHMPW20b}.
  However, if $P$ is realized by $V \subset \K^{\mathbf n}$ with $\K$ of characteristic 0, then our work gives an independent proof of top-heaviness for $\mathcal L_P$ via arguments similar to those contained in \cite{HW17,BE09}.

  In outline:
  By \cite[Proposition 6.8.11]{O11}, we may assume that $\K = \C$.
  Let $\IH^i(Y_V)$ be the degree $i$ intersection cohomology of $Y_V$, and let $H^i(Y_V)$ be the degree $i$ singular cohomology.
  Arguments of \cite{HW17,BE09} show that there is an injection of $H^*(Y_V)$-modules $H^*(Y_V) \to \IH^*(Y_V)$.
  Intersection cohomology has the Hard Lefschetz property for any ample class $L \in H^{2}(Y_V)$, meaning that the multiplication maps
  \[
\IH^i(Y_V) \to \IH^{d-i}(Y_V), \quad \alpha \mapsto L^{d-2i} \alpha,
  \]
  are injective when $i \leq d/2$.
  These maps preserve the submodule $H^*(Y_V) \subset \IH^*(Y_V)$, so they restrict to injections $H^i(Y_V) \to H^{d-i}(Y_V)$.
  \cref{cohomology2} states that $\dim H^k(Y_V) = |\mathcal L_P^k|$, so taking dimensions yields top-heaviness for $\mathcal L_P$.
\end{rmk}

\section{Problems}\label{sec:problems}
\subsection{} We produce Schubert varieties for polymatroids $P$ realizable over a field of characteristic 0.
In this case, our work provides a geometric proof of the top-heavy property for $\mathcal L_P$.
The proof of top-heaviness for realizable matroids \cite{HW17} relies on being able to construct Schubert varieties of hyperplane arrangements in positive characteristic.
\begin{question}
  \revision{Given a polymatroid $P$ realizable over a field of positive characteristic, construct a variety with an affine paving whose poset of strata if the lattice of combinatorial flats, thereby proving top-heaviness of $\mathcal L_P$.}
  \end{question}

\subsection{} Schubert varieties of hyperplane arrangements are always normal, but this is untrue for Schubert varieties of subspace arrangements.
\begin{zb}\label{zb:nonnormal}
  Let $V$ be cut out by the equation $a_{11} + a_{12} = a_{21} + a_{22}$ in $\C^2 \times \C^2$.
  Then $Y_V$ is the hypersurface in $\P^2 \times \P^2$ defined by
  \[
    b_{20}^2( b_{10}b_{11} + b_{10}b_{12} - \tfrac{1}{2} b_{11}^2)
    = b_{10}^2 (b_{20} b_{21} + b_{20} b_{22} - \tfrac12 b_{21}^2).
  \]
  Setting $b_{ij/k\ell} := b_{ij} / b_{k\ell}$, let $U$ be, for example, the affine chart on which $b_{12} \neq 0$ and $b_{22} \neq 0$, with coordinate ring
  \[
    \mathcal O(U) = \frac{\K[b_{10/12}, b_{11/12}, b_{20/22}, b_{21/22}]}
    {\langle b_{20/22}^2(b_{10/12} b_{11/12} + b_{10/12} - \tfrac12 b_{11/12}^2) - b_{10/12}^2(b_{20/22}b_{21/22} + b_{20/22} - \tfrac12 b_{21/22}^2) \rangle}.
  \]
  The rational function
  \[
    f = \frac{b_{20/22}}{b_{10/12}} ( b_{10/12}b_{11/12} + b_{10/12} - \tfrac{1}{2} b_{11/12}^2)
  \]
  satisfies the monic polynomial
  \[ X^2 - ( b_{10/12}b_{11/12} + b_{10/12} - \tfrac{1}{2} b_{11/12}^2)(b_{20/22} b_{21/22} + b_{20/22} - \tfrac12 b_{21/22}^2) \in \mathcal O(U)[X], \]
  yet is not a regular function on $U$. 
  Nevertheless, in general we conjecture that $Y_V \subset \prod_{i=1}^N \P^{n_i}$ is normal when the subspaces $V \cap \ker(\pi_{E \setminus i})$, $1 \leq i \leq N$, span $V$.
\end{zb}
\begin{question}
  Determine when Schubert varieties of subspace arrangements are normal.
\end{question}

\subsection{} The ideals defining Schubert varieties of hyperplane arrangements in $(\P^1)^N$ are studied in \cite{AB16}, and found to be highly tractable: multidegrees, multigraded Betti numbers, and initial ideals can all be understood explicitly in terms of the combinatorics of the relevant matroid.
We believe these objects will also admit combinatorial descriptions for Schubert varieties of subspace arrangements.
A major obstacle to verifying this is that we do not know their defining equations.
\begin{question}
  Find defining equations for Schubert varieties of subspace arrangements.
\end{question}
Unlike those of Schubert varieties of hyperplane arrangements, equations for Schubert varieties of subspace arrangements cannot be obtained by mere homogenization of circuits, even in relatively simple cases.
\begin{zb}\label{zb:equations}
  Let $V$ be the line in $\C^3$ spanned by $(1,1,1)$, which represents the polymatroid $P$ of rank 1 on one element.
\revision{A minimal set of polynomials defining $Y_V \subset \P^3$ is:
  \begin{gather*}
    6\,x_{0}b_{12}+4\,x_{1}b_{12}+2\,b_{12}^{2}-6\,x_{0}x_{3}-3\,x_{1}x_{3} \\
    x_{1}^{2}-2\,x_{1}b_{12}-2\,b_{12}^{2}+3\,x_{1}x_{3} \\
    3\,x_{0}x_{1}+5\,x_{1}b_{12}+4\,b_{12}^{2}-3\,x_{0}x_{3}-6\,x_{1}x_{3}.
  \end{gather*}}
  \revision{The ideal generated by these polynomials is not the ideal of homogenized circuits of $V$.
  However, observe that there are three minimal equations of degree 2, and that the matroid lift $\widetilde P$ with respect to $\mathbf n = (3)$ is a uniform matroid of rank 1 on three elements, which has three circuits of size 2.

  In general, we observe in examples that if $Y_V \subset \P^{\mathbf n}$, then the number of
  minimal equations of a given multidegree $\mathbf d$ is equal to the number of circuits of the lift $\widetilde P$ with respect to $\mathbf n$ that represent $\mathbf d$.}
\end{zb}
Along similar lines, \cite[Theorem 1.5]{AB16} gives a combinatorial formula for the multigraded Betti numbers of the ideal of a Schubert variety of hyperplane arrangement.
\begin{question}
  Find a combinatorial formula for the multigraded Betti numbers of the ideal of a Schubert variety of subspace arrangement.
\end{question}
\subsection{}
The Chow class of a Schubert variety of subspace arrangement $Y_V \subset \P^{\mathbf n}$ may not be multiplicity-free, and the coefficients $c_{\mathbf b}$ of $[Y_V]$ may depend on $\mathbf n$.
We conjecture the following formula:
\begin{conj}
  If $Y_V \subset \P^{\mathbf n}$, then \revision{when $\mathbf b$ is a basis,} the coefficient of $[\P^{\mathbf b}]$ in $[Y_V]$ is $c_{\mathbf b} = \binom{n_1}{b_1}\binom{n_2}{b_2} \cdots \binom{n_N}{b_N}$.
\end{conj}
\subsection{}
The singularities of Schubert varieties of hyperplane arrangements can be resolved in a ``universal fashion'' using the \word{stellahdral variety} of \cite{BHMPW20}.
\begin{question}
  Resolve the singularities of Schubert varieties of subspace arrangements. Is there a ``universal'' resolution?
\end{question}

\subsection{}
The \word{$Z$-polynomial} of a matroid $M$ is a combinatorial invariant defined in \cite{Zpoly}.
When $M$ is realized by $V \subset \K^N$, the $Z$-polynomial is the Poincar\'e polynomial of the intersection cohomology of the Schubert variety of hyperplane arrangement $Y_V$.
\begin{question}
  Compute the intersection cohomology Betti numbers of a Schubert variety of subspace arrangement.
  Are they combinatorially determined?
\end{question}
An analogous question can be asked for the \word{Kazhdan-Lusztig polynomial} of \cite{KLpoly}, which in the case of a realizable matroid corresponds to the local intersection cohomology of $Y_V$ at the point $(\infty, \ldots, \infty)$.
\bibliographystyle{amsalpha}
\bibliography{bibliography}
\end{document}